\documentclass[reqno]{amsart}
\usepackage{amscd}
\usepackage{amssymb}
\usepackage[dvips]{graphics}
\usepackage[margin=1in,includeheadfoot]{geometry}
\usepackage{url}
\usepackage{hyperref}
\usepackage{cite}
\usepackage{color}

\def\beq{\begin{equation}}
\def\eeq{\end{equation}}
\def\ba{\begin{array}}
\def\ea{\end{array}}

\def\R{\mathbb R}

\newtheorem{thm}{Theorem}[section]
\newtheorem{lm}[thm]{Lemma}
\newtheorem{prop}[thm]{Proposition}

\theoremstyle{definition}
\newtheorem{rem}[thm]{Remark}

\newtheorem{df}[thm]{Definition}

\theoremstyle{remark}

\begin{document}
\pagestyle{plain}
\title{Boundary pointwise regularity for the Poisson problem on uniform domain}
\author{Tianyu Guan\\
 \small{School of Mathematical Sciences, Shanghai Jiao Tong University}\\
 \small{Shanghai, China}\\\\
\small{ Lihe Wang}\\
 \small{Department of Mathematics, University of Iowa, Iowa City, IA, USA;\\ School of Mathematical Sciences, Shanghai Jiao Tong University}\\
 \small{Shanghai, China}\\\\
\small{Chunqin Zhou}\\
 \small{School of Mathematical Sciences, Shanghai Jiao Tong University}\\
 \small{Shanghai, China}}

\footnote{The third author was partially supported by NSFC Grant 12571223, 12031012 and STCSM Grant 24ZR1440700.}

\begin{abstract} In this paper, we study the boundary pointwise regularity for the Poisson problem on domains with rough boundaries, specifically uniform domains. In general, it is not straightforward to define weak solutions for non-zero boundary data on such domains. To address this, we introduce a novel definition of weak solutions tailored to the setting of uniform domains. Remarkably, this definition allows for the analysis of the regularity of weak solutions. In particular, by establishing an energy inequality, we prove the boundary pointwise $C^\alpha$ regularity by using compactness methods under the admissible condition. Furthermore,  by exploiting the the linear structure of solutions with respective to the harmonic functions, we establish boundary pointwise $C^{1,\alpha}$ and $C^{2,\alpha}$ regularities when the boundary data and the domain boundary are pointwise $C^{1,\alpha}$ and $C^{2,\alpha }$, respectively.
\end{abstract}

\maketitle

{\bf Keywords: Uniform domain, the Poisson problem, Boundary pointwise $C^{k,\alpha}$-regularity }

\section{Introduction}

Let $\Omega\subset \R^n$ be a bounded domain with $n\geq 3$.  For the divergence form elliptic boundary problem
\begin{equation}\label{1.1}
\left\{
\begin{array}{rcll}
-D_i(a_{ij}D_ju)+b_i(x)D_iu+c(x)u &=&f+D_i f^i\qquad&\text{in}~~\Omega,\\
u&=&g\qquad&\text{on}~~\partial\Omega, \\
\end{array}
\right.
\end{equation}
since the trace of $u\in H^1(\Omega)$ on $\partial \Omega$ is well defined when  $\partial \Omega\in C^{0,1}$, we usually say $u\in H^{1}(\Omega)$ is a weak solution to (\ref{1.1}) if it satisfies
\begin{equation*}
\int_{\Omega}a_{ij}(x)D_juD_i\varphi+b_i(x)D_iu\varphi+c(x)u\varphi dx=\int_{\Omega}f\varphi-f^iD_i\varphi dx
\end{equation*}
for all test function $\varphi\in H_0^{1}(\Omega)$, and $ u-g=0$ on $\partial\Omega$ in the sense of the trace. To some extent, we want to consider more general domains and consider the situations that beyond $u\in H^1(\Omega)$.

Under suitable conditions--such as uniform ellipticity of $a_{ij}(x)$, boundedness of $a_{ij}(x)$, $b_i(x), c(x)$, $f\in L^{\frac{q}{2}}(\Omega)$, $f^i\in L^{q}(\Omega)$, $g\in H^{1}(\Omega)\cap C^{\alpha}(\overline{\Omega}) $ with $q>n$ and $\alpha\in (0,1)$, and a uniform exterior cone condition on $\partial \Omega$--one can obtain the classical global $C^\alpha(\overline{\Omega})$ regularity  for solutions. It is known that the  uniform exterior condition can be weakened to  a $\theta$-admissible condition for  $\theta>0$, namely,
\begin{equation*}|\Omega^{c}\cap B_{r}(x_0)|\geq \theta|B_{r}(x_0)|,\nonumber
\end{equation*}
 for any $x_0\in\partial{\Omega}$ and $r>0$  (see Theorem 8.29 in \cite{DT} when $g=0$).

As demonstrated in the definition of weak solutions and their subsequent proofs, it is observed that Lipschitz continuity ($\partial \Omega\in C^{0,1}$
) is the minimal regularity condition required for the boundary H\"older regularity of weak solutions to divergence-form elliptic equations. This is, in part, due to Calder\'on-Stein's Theorem, which states that every Lipschitz domain is an extension domain for Sobolev spaces.

Jones \cite{Jo} showed that every $L$-uniform domain is  an extension domain for Sobolev spaces. Moreover, a finitely connected planar domain is an extension domain if and only if it is an $L$-uniform domain. Jones conjectured that $L$-uniform domains are the worst possible domains for which classical function-theoretic properties are the same as those of the Euclidean upper half-space.

In this paper, for simplicity, we focus on the boundary pointwise regularity for the following Dirichlet problem
\begin{equation}\label{1}
\left\{
\begin{array}{rcll}
-\Delta u&=&f\qquad&\text{in}~~\Omega,\\
u&=&g\qquad&\text{on}~~\partial\Omega, \\
\end{array}
\right.
\end{equation}
where $\Omega$ is $L$-uniform domains, $f\in L^2(\Omega)$ and $g\in C(\partial\Omega)$. More precisely we define:

\begin{df}\label{uni} For $L>1$, we call a domain $\Omega\subset \mathbb{R}^{n}$ an $L$-uniform domain if, for any pair of points $x_1 , x_2\in \Omega$, there exists a rectifiable curve $\gamma$ :$[0,1]\rightarrow \Omega$, s,t. $\gamma(0)=x_1,\gamma(1)=x_2$, and  its 1-dimensional Hausdorff measure\\
\begin{equation}\label{2}
 \mathcal{H}^{1}(\gamma)\leq L|x_1-x_2|,
 \end{equation}
 and
\begin{equation}\label{3} d(\gamma(t),\partial{\Omega})\geq \frac 1L \min\limits_{t\in[0,1]} \{|\gamma(t)-x_1|, |\gamma(t)-x_2|\}.
\end{equation}
\end{df}

$L$-uniform domains play a central role in analysis and geometry. Roughly speaking, they are connected in some quantitative sense, and for each such curve $\gamma(t)$, there exists a tube $T$ containing $\gamma(t)$ with width at $\gamma(t)$ of order $\min\{|x_1-\gamma(t)|,|x_2-\gamma(t)|\}$. The class of $L$-uniform domains includes convex domains in a ball, uniform Lipschitz domains, quasi-extremal distance domains, and the classical snowflake domain. Importantly, $L$-uniform domains can be highly nonrectifiable.

If $n-1\leq \alpha<n$, one can construct a domain $\Omega\subset \R^n$ such that $\Omega$ is an $L$-uniform domain and $\mathcal{H}^{\alpha}(U\cap\partial \Omega)>0$ with some neighborhood $U$ at a boundary point. If an $L$-uniform domain $\Omega$ has finite perimeter, the trace of $u\in H^1(\Omega)$ can be defined on the reduced boundary $\partial^*\Omega$ in a suitable sense (see \cite {AFP, DLW2023}). However, there is no general trace theorem for arbitrary $L$-uniform domains.

The remarkable feature of $L$-uniform domains is that they are Sobolev extension domains (see \cite{DLW2023, YS2006}). Consequently, the extension theorem and Sobolev inequalities hold on such domains, which are crucial for the regularity theory for both elliptic and parabolic equations.

A natural question arises: how to define solutions to boundary value problems in general domains with rough boundaries, where the trace embedding theorem does not apply?

For non-divergence form elliptic operators in a bounded domain
$$
L=M+c(x)=a_{ij}(x)\partial_{ij}+b_i(x)\partial _i+c(x),
$$ with $a_{ij}(x)$ continuous and uniformly elliptic in $\Omega$, $a_{ij}(x)$, $b_i(x)$, $c(x)$ bounded, it is classical that $L$ acts on functions $u\in C^2(\Omega)$ or $u\in W^{2,n}_{loc}(\Omega)$, and the maximum principle holds for $L$ in $\Omega$ if
$$
Lu\geq 0~~~\text{ in } \Omega
$$
and
\begin{equation}\label{cc}
\limsup_{x\rightarrow \partial \Omega} u(x)\leq 0
\end{equation}
imply $u(x)\leq 0$ in $\Omega$ when $c(x)\leq 0$. Moreover, if the coefficients and  $\partial \Omega$ possess some regularity, the maximum principle holds if and only if the principal eigenvalue of $-L$ for the Dirichlet problem is positive.

In \cite{BNV}, Berestycki, Nirenberg and Varadhan observed that in an arbitrary bounded domain $\Omega$, the condition (\ref{cc}) is in nondivergent form and solutions are locally smooth, in general, one cannot prescribe boundary values or know the precise behavior of solutions at every point of $\partial \Omega$. Instead,  they introduced a refined version of the maximum principle. They selected a barrier function $u_0>0$ in $\Omega$ satisfying $Mu_0=-1$, vanishing in a suitable sense on $\partial \Omega$. Then, instead of $(\ref{cc})$, they assumed that for every sequence $x_j\rightarrow\partial \Omega$ with $u_0(x_j)\rightarrow 0$,
$$
\limsup  u(x_j)\leq 0,
$$ and further showed that $L$ satisfies the refined maximum principle for the bounded above function in general domains. They also proved that the refined maximum principle holds if and only if the principal eigenvalue of $-L$ is positive.

For Dirichlet problem of the divergence form elliptic equation
\begin{equation*}
\left\{
\begin{array}{rcll}
-D_i(a_{ij}D_j u)&=&D_i f^i\qquad&\text{in}~~\Omega,\\
u&=&g\qquad&\text{on}~~\partial\Omega, \\
\end{array}
\right.
\end{equation*}
when $g=0$ and $\Omega$ is the general bounded domain, the weak solution is typically defined as a function $u\in H^1(\Omega)$ satisfying
\begin{equation*}
\int_{\Omega}a_{ij}D_{i}uD_j \varphi dx=\int_{\Omega}f^iD_i\varphi dx
\end{equation*}
for all test function $\varphi\in H_0^{1}(\Omega)$, with $u$'s $0-$extension is in $H^1(\R^n)$. Clearly, such a weak solution satisfies the energy inequality. In  \cite{BW04}, Byun and Wang  successfully proved $W^{1,p}$ estimates using the energy inequality for this kind of definition of weak solutions when $\Omega$ is a Reifenberg flat domain-a class of domains with rough boundaries. Later, Byun, Wang, and Zhou extended this result to quasi-linear elliptic equations (see \cite{BWZ}).

In \cite{LZLH}, Lian-Zhang-Li-Hong studied the Dirichlet problem of the divergence form elliptic equation
\begin{equation*}
\left\{
\begin{array}{rcll}
-D_i(a_{ij}D_j u)&=& f\qquad&\text{in}~~\Omega,\\
u&=&g\qquad&\text{on}~~\partial\Omega, \\
\end{array}
\right.
\end{equation*}
where $\Omega$ satisfies a geometric condition at $x_0\in \partial \Omega$:  there exist a constant $0<\nu <1$, a quasi-geometric sequence $\{r_k\}_{k=0}^{\infty}$ and a sequence $\{y_k\}_{k=0}^{\infty}$ with $y_k \in \partial B(x_0, r_k)$ such that
$$
\partial B(x_0,r_k)\cap B(y_k,\nu r_k)\subset \Omega^c.
$$
Assume the weak solution is continuous up to the boundary, $a_{ij}$ is uniformly elliptic, $f\in L^P(\Omega)$ with $p>\frac n2$ and $g\in C^{\alpha}(x_0)$, they proved that $u$ is $C^{\beta}$ at $x_0$ with
$$
|u(x)-u(x_0)|\leq C|x-x_0|^{\beta}(||u||_{L^\infty(\Omega\cap  B_1(x_0))}+||f||_{L^p(\Omega\cap B_1(x_0))}+[g]_{C^{\alpha}(x_0)})
$$ for all $x\in \Omega \cap  B_1(x_0)$,  where $0<\beta \leq \min (2-\frac np, \alpha)$.  Clearly, this boundary H$\ddot{o}$lder regularity begins with $C^0$ solutions.

However, when $g\neq 0$ and a typical continuous function or even for $C^\alpha$ with $\alpha<\frac 12$, the requirement of $u\in H^1(\Omega)$ is too strong. Instead, we propose a definition based on subsolutions of $(u-\sup_{\Omega\cap \partial B_r(x_0)}g)^+$ and $(u-\inf_{\Omega\cap \partial B_r(x_0)}g)^-$, which satisfy a boundary energy inequality. From these energy estimates, we establish the boundary pointwise $C^\alpha$ regularity via the compactness method and perturbation techniques. Moreover, this definition satisfies a linear structure with respective to harmonic functions, enabling us to prove $C^{1,\alpha}$ and $C^{2,\alpha}$ boundary regularities using the iterative technique. For simplicity, we only deal with the Poisson  problem.  For the existence of such solutions, we are preparing in our later paper.

\begin{df}\label{uni} Let $\Omega\subset\mathbb{R}^n$ be a bounded domain, $f\in L^2(\Omega)$, $g\in L^{\infty}({\partial \Omega})$. We say a function $u$ is a weak solution of
\begin{equation*}
\left\{
\begin{array}{rcll}
-\Delta u&=&f&\text{in}~~\Omega,\\
u&=&g\qquad&\text{on}~~\partial\Omega, \\
\end{array}
\right.
\end{equation*}
if: (i) $u\in H^{1}_{loc}(\Omega)$ is a local solution in $\Omega$, i.e.,
\begin{equation}\label{eq1}
\displaystyle\int_{\Omega}\nabla u\cdot\nabla \varphi dx=\displaystyle\int_{\Omega}f\cdot\varphi dx,\quad \forall~\varphi\in C_{0}^{\infty}(\Omega),
\end{equation}
(ii) for any $ 0<r\leq diam(\Omega)$ and $x_0\in\partial\Omega$, $(u-\sup\limits_{B_r(x_0)\cap \partial\Omega}g)^{+}$'s 0-extension in $B_r(x_0)\cap\Omega^{c}$ is in $H^{1}(B_r(x_0))$, and it is a weak subsolution in $B_r(x_0)$, i.e.,
\begin{equation}\label{eq2}
\displaystyle\int_{B_r(x_0)}\nabla[(u-\sup\limits_{B_r(x_0)\cap\partial\Omega}g)^+]\cdot\nabla \varphi dx\leq \displaystyle\int_{B_r(x_0)}f^+\varphi dx,\quad \forall~\varphi\in C_{0}^{\infty}(B_r(x_0)) \text { and } \varphi\geq 0.
\end{equation}
(iii) for any $0<r\leq diam(\Omega)$ and $x_0\in\partial\Omega$, $(u-\inf\limits_{B_r(x_0)\cap \partial\Omega}g)^{-}$'s 0-extension  in $B_r(x_0)\cap\Omega^{c}$ is in $H^{1}(B_r(x_0))$, and it is a weak subsolution in $B_r(x_0)$, i.e.,
\begin{equation}\label{eq3}
\displaystyle\int_{B_r(x_0)}\nabla[(u-\inf\limits_{B_r(x_0)\cap\partial\Omega}g)^-]\cdot\nabla \varphi dx\leq \displaystyle\int_{B_r(x_0)}f^-\varphi dx,\quad \forall~\varphi\in C_{0}^{\infty}(B_r(x_0))  \text { and } \varphi\geq 0.
\end{equation}
\end{df}

 In the above definition, we have extended $f(x)$ to zero in $\Omega^c$. In the sequent, we say $u$ is a weak solution of ($\ref{1}$) always in the sense of the definition $\ref{uni}$. The main theorem is stated as follows.

\begin{thm}\label{thm1}Let $\Omega$ be a $\theta$-admissible and $L$-uniform domain with $diam(\Omega)\geq C_0$ and $0\in\partial\Omega$. Suppose $u$ is a weak solution of $(\ref{1})$. Then there exists $0<\alpha_0<1$ such that for all $ 0<\alpha<\alpha_0$, if $f$ is $ C^{-2,\alpha}$ at $0$ in the $L^{2}$ sense, $g$ is $ C^{\alpha}$ at $0$, then  $u$ is $C^{\alpha}$ at $0$ in the $L^{2}$ sense, and
$$
[u]_{\mathcal{L}^{2,\alpha}(0)}\leq C\left(\displaystyle\frac{1}{|B_{1}|}\int_{\Omega\cap B_{1}}u^{2}dx
\right)^{\frac{1}{2}}+C\left([f]_{\mathcal{L}^{2,-2+\alpha}(0)}+||g||_{C^{\alpha}(0)}\right).
$$
\end{thm}

If the boundary $\partial \Omega$ has the better  pointwise geometric property, such as  $\partial\Omega\in C^{k,\alpha}(x_0)$ for $k=1, 2$, we can establish the pointwise $C^{k,\alpha}$ boundary regularity for weak solutions.

\begin{thm}\label{thm2}Let $\Omega$ be $C^{1,\alpha}$ at $0\in\partial\Omega$. Suppose $u$ is a weak solution of $(\ref{1})$. For any $0<\alpha<1$, if $f$ is $ C^{-1,\alpha}$ at $0$ in the $L^{2}$ sense, $g$ is $ C^{1,\alpha}$ at $0$, then  $u$ is $C^{1,\alpha}$ at $0$ in the $L^{2}$ sense, and
$$
[u]_{\mathcal{L}^{2,1+\alpha}(0)}\leq C\left(\displaystyle\frac{1}{|B_{1}|}\int_{\Omega\cap B_{1}}u^{2}dx
\right)^{\frac{1}{2}}+C\left([f]_{\mathcal{L}^{2,-1+\alpha}(0)}+||g||_{C^{1,\alpha}(0)}\right).
$$
\end{thm}

\begin{thm}\label{thm3}Let $\Omega$ be $C^{2,\alpha}$ at $0\in\partial\Omega$. Suppose $u$ is a weak solution of $(\ref{1})$. For any $0<\alpha<1$, if $f$ is $ C^{\alpha}$ at $0$ in the $L^{2}$ sense, $g$ is $ C^{2,\alpha}$ at $0$, then $u$ is $C^{2,\alpha}$ at $0$ in the $L^{2}$ sense, and
$$
[u]_{\mathcal{L}^{2,2+\alpha}(0)}\leq C\left(\displaystyle\frac{1}{|B_{1}|}\int_{\Omega\cap B_{1}}u^{2}dx
\right)^{\frac{1}{2}}+C\left(||f||_{\mathcal{L}^{2,\alpha}(0)}+||g||_{C^{2,\alpha}(0)}\right).
$$
\end{thm}

\begin{rem}
Once H\"older continuity is established in the $L^2$ sense, in conjunction with the standard interior H\"older regularity, classical H\"older regularity follows from the Campanato embedding theorem (see Section 2).
\end{rem}

\begin{rem}[Divergence form problem with small BMO coefficients]
Although we state and prove the main results for the Poisson problem, the definition of solutions and the compactness method used in this paper are not restricted to the Poisson problem. The same argument can be adapted to divergence form elliptic boundary problem
\begin{equation*}
\left\{
\begin{array}{rcll}
-D_i(a_{ij}D_j u)&=& f\qquad&\text{in}~~\Omega,\\
u&=&g\qquad&\text{on}~~\partial\Omega, \\
\end{array}
\right.
\end{equation*}
where \(A(x)=(a_{ij}(x))\) satisfies
\[
\lambda |\xi|^2\le A(x)\xi\cdot \xi\le \Lambda |\xi|^2,
\qquad \xi\in\mathbb R^n, ~~x\in \Omega,
\]
and has sufficiently small mean oscillation. More precisely, one may assume that \(A\) is \((\delta,R)\)-vanishing in the sense that
\[
\sup_{0<\rho\le R}\sup_{y\in\mathbb R^n}
\sqrt{\frac{1}{|B_\rho(y)\cap{\Omega}|}\int_{B_\rho(y)\cap{\Omega}}
|A(x)-\overline{A}_{B_\rho(y)\cap {\Omega}}|^2\,dx}
\le \delta,
\]
for \(\delta>0\) sufficiently small and some $R>0$.

Indeed, under this assumption, after rescaling around a boundary point, the coefficients are close in average to a constant matrix. The approximation lemma used for the Poisson equation is then replaced by the corresponding compactness lemma comparing \(u\) with the solution of the constant-coefficient problem
\begin{equation*}
\left\{
\begin{array}{rclll}
-\operatorname{div}(A_0\nabla h) &=& 0~~\text{ in } \tilde{\Omega}_{\frac 12},\\
h&=& 0~~\text{ on }  \partial_w \tilde{\Omega}_{\frac 12},
\end{array}
\right.
\end{equation*}
where $\tilde{\Omega}$ is $L$-uniform domain and $\theta$-admissible domain.


Consequently, the boundary  H\"older estimates proved in this paper remain valid for such divergence form elliptic boundary problem with sufficiently small BMO coefficients. We restrict the presentation to the Poisson problem in order to avoid additional technical notation and to emphasize the main boundary compactness mechanism.
\end{rem}

To study the boundary behavior for solutions, one typically flattens the curved boundary via a transformation, which introduces low-order terms and variable coefficients. However, this approach often requires additional smoothness on the boundary and boundary value, and such proofs cannot be extended to our setting.

In our case, the main difficulty lies in the fact that $\partial \Omega$ is far from being a graph domain, so we cannot flatten it. For viscosity solutions of non-divergence form elliptic problems, Ma and Wang established boundary pointwise $C^{1,\alpha} $ regularity using a barrier argument iteration (see \cite{MW}). Later, perturbation and compactness techniques were used to prove pointwise $C^{1,\alpha} $ and $C^{2,\alpha} $ regularity under corresponding pointwise geometric boundary conditions (see \cite{LZ, LZ2020, SS, Wa}).

Our approach is heavily influenced by \cite{Ca,Ca1989,BW04, LZ, LZ2020}. We begin with a definition of weak solutions on domains with the irregular boundaries, establish energy estimates near the boundary,  and use compactness techniques to obtain the approximate solutions on new irregular domains. The error between the original and new domains is controlled because the energy on narrow regions is small in some order. For divergence form elliptic boundary problems, our results are new. In particular, when using the compactness method, we establish the compactness of a sequence of $L$-uniform and $\theta$-admissible domains. Compared to \cite{LW, DLW2023}, we provide a new of proof using PDE techniques. To prove $C^{1,\alpha}$ and $C^{2,\alpha}$ boundary regularities under the assumption that the boundary data and domain boundary are pointwise $C^{1,\alpha}$ and $C^{2,\alpha }$, respectively, the key step is to establish the linear structure of weak solutions (in the sense of Definition \ref{uni}) with respective to harmonic functions. It is worth noting that it is the first work to define weak solutions on the irregular domains with non-zero boundary data and to study the boundary pointwise regularity for such solutions.

\section{Preliminaries}
In this section, we collect some notation and standard facts that will be used throughout the paper. We first recall the geometric assumptions on the domains, then state the basic analytic tools, including the Harnack inequality, boundary H\"older estimates, and a pointwise characterization for a function.

Throughout the paper, for $x_0\in \mathbb R^n$ and $r>0$, we write
\[
B_r(x_0)=\{x\in \mathbb R^n: |x-x_0|<r\},\qquad B_r=B_r(0),
\]
and
\[
\Omega_r(x_0)=\Omega\cap B_r(x_0),\qquad \Omega_r=\Omega\cap B_r.
\]
If $x_0\in \partial\Omega$, we denote the boundary portion by
\[
\partial_w\Omega_r(x_0)=\partial\Omega\cap B_r(x_0).
\]
When no confusion is possible, we omit the center $x_0$.

 Also the notation $E\Delta F:=(E\backslash F)\cup (F\backslash E)$ represents the symmetric difference of two sets $E$ and $F$. The letter $C$  denotes a universal constant depending on known quantities such as $n$, $p$, the geometric conditions on $\partial \Omega$. The constant $C$ may vary from line to line.
\subsection{Geometric assumptions}

We first recall the definition of uniform domains.

\begin{df}[$L$-uniform domain]
Let $L>1$. A domain $\Omega\subset \mathbb R^n$ is called an $L$-uniform domain if for any $x_1,x_2\in \Omega$, there exists a rectifiable curve $\gamma:[0,1]\to \Omega$ such that $\gamma(0)=x_1$, $\gamma(1)=x_2$, and
\[
\mathcal H^1(\gamma)\leq L|x_1-x_2|,
\]
and
\[
d(\gamma(t),\partial\Omega)\geq \frac1L
\min\{|\gamma(t)-x_1|,|\gamma(t)-x_2|\},
\qquad t\in[0,1].
\]
\end{df}

We also use the following measure density condition.

\begin{df}[$\theta$-admissible domain]
Let $\theta\in(0,1)$. A bounded domain $\Omega\subset \mathbb R^n$ is called $\theta$-admissible if
\[
|\Omega^c\cap B_r(x_0)|\geq \theta |B_r|
\]
for every $x_0\in \partial\Omega$ and every $0<r<\operatorname{diam}(\Omega)$.
\end{df}



It follows from the quantitative corkscrew property of uniform domains, the following elementary consequences of a uniform domain  will be used in the sequel. See, for instance, Proposition 2.4 in \cite{DLW2023}.

\begin{lm}
Let $\Omega$ be an $L$-uniform domain. Then there exists a constant $c=c(n,L)>0$ such that for every $x_0\in\partial\Omega$ and every $0<r<\operatorname{diam}(\Omega)$, there exists a ball
\[
B_{cr}(y)\subset \Omega\cap B_r(x_0).
\]
Moreover, if $\operatorname{diam}(\Omega)\geq C_0>0$, then $\Omega$ contains an interior ball whose radius is bounded from below by a positive constant depending only on $n,L$ and $C_0$.
\end{lm}


\begin{prop}[Interior density of uniform domains]
Let $\Omega$ be an $L$-uniform domain. If $\operatorname{diam}(\Omega)\geq C_0>0$, then there exists a constant
$c=c(n,L, \frac {\operatorname{diam}(\Omega)}{C_0})>0$ such that
\[
|\Omega\cap B_r(x_0)|\geq c r^n
\]
for every $x_0\in\partial\Omega$ and every
$0<r<\operatorname{diam}(\Omega)$.
\end{prop}

We shall also use the fact that uniform domains are Sobolev extension domains. More precisely, if $\Omega$ is an $L$-uniform domain, then for each $1\leq p<\infty$ there exists a bounded linear extension operator
\[
E:W^{1,p}(\Omega)\to W^{1,p}(\mathbb R^n),
\]
with operator norm depending only on $n,p$ and $L$. In particular, Sobolev and Poincar\'e inequalities hold on $L$-uniform domains with constants depending only on $n,p$ and $L$; see \cite{Jo,DLW2023}.

\
\

Next we give the pointwise geometric conditions for the boundary, adapted from \cite{LZ2020}.
\begin{df} Let $x_0\in\partial\Omega$. For a positive integer $k$, we say that $\partial\Omega$ is $C^{k,\alpha}$ at $x_0$, denoted by $\partial\Omega\in C^{k,\alpha}(x_0)$, if there exist $r_0>0$, a coordinate system $\{x_1,\dots,x_n\}$(obtained by translating and rotating the original coordinates), and a $k-th$ order polynomial $P(x')$ with $P(0)=0$, $DP(0)=0$, such that $x_0=0$ in this coordinate system, and
    \begin{equation}\label{d1}
        B_{r_0}\cap \{(x',x_n)|x_n>P(x')+K|x'|^{k+\alpha}\}\subset B_{r_0}\cap \Omega,
    \end{equation}
and
\begin{equation}\label{d2}
     B_{r_0}\cap \{(x',x_n)|x_n<P(x')-K|x'|^{k+\alpha}\}\subset B_{r_0}\cap \Omega^{c}.
\end{equation}
We define
$$
[\partial\Omega]_{C^{k,\alpha}(x_0)}=\inf\{K|(\ref{d1}), (\ref{d2})~ \text{hold}\},
$$
and
$$
  \Vert \partial\Omega\Vert_{C^{k,\alpha}(x_0)}=\Vert P\Vert + [\partial\Omega]_{C^{k,\alpha}(x_0)},
$$
where $\Vert P\Vert=\sum_{m=0}^{k}|D^m P(x_0)|$.
\end{df}

\begin{rem}
From the definition, we see  that $\partial \Omega$ may not be represented as the graph of a function near $x_0$.
\end{rem}

\subsection{Harnack inequality and interior estimates}

We recall the standard Harnack inequality for nonnegative harmonic functions.

\begin{lm}[Harnack inequality]
Let $u\geq0$ be harmonic in $B_{2r}(x_0)$. Then
\[
\sup_{B_r(x_0)}u\leq C\inf_{B_r(x_0)}u,
\]
where $C$ depends only on $n$.

\end{lm}

Consequently, if a nonnegative harmonic function is defined in a domain and two points can be connected by a chain of balls whose doubled balls remain inside the domain, then the values of the function at the two points are comparable. More precisely, if
\[
B_{2\rho_j}(x_j)\subset \Omega,\qquad j=1,\dots,N,
\]
and the balls form a Harnack chain connecting $x$ to $y$, then
\[
u(x)\leq C^N u(y),
\]
where $C=C(n)$.

We shall use this form of the Harnack inequality in combination with the uniform curve condition. In particular, the uniformity of the domain gives Harnack chains whose length is controlled only by $n$ and $L$ at comparable scales.

We also recall the standard local boundedness and H\"older estimates. If $u$ is harmonic in $B_{2r}(x_0)$, then there exist constants $C>0$ and $\alpha\in(0,1)$, depending only on $n$, such that
\[
\|u\|_{L^\infty(B_r(x_0))}
\leq C\left(\frac{1}{|B_{2r}(x_0)|}\int_{B_{2r}(x_0)}u^2\,dx\right)^{1/2},
\]
and
\[
[u]_{C^\alpha(B_r(x_0))}
\leq C r^{-\alpha}
\left(\frac{1}{|B_{2r}(x_0)|}\int_{B_{2r}(x_0)}u^2\,dx\right)^{1/2}.
\]

\subsection{Boundary H\"older estimates}

The exterior density condition implies boundary H\"older regularity for solutions vanishing on the boundary. We shall use the following standard estimate.

\begin{lm}[Boundary H\"older estimate]
Let $\Omega$ be a $\theta$-admissible domain, $x_0\in\partial\Omega$, and $0<r<\operatorname{diam}(\Omega)$. Suppose that $v\in H^1(\Omega_{2r}(x_0))$ satisfies
\[
-\Delta v=0\qquad \text{in }\Omega_{2r}(x_0),
\]
and that the zero extension of $v$ belongs to $H^1(B_{2r}(x_0))$. Then there exist constants $C>0$ and $\alpha\in(0,1)$, depending only on $n$ and $\theta$, such that
\[
|v(x)|\leq C\left(\frac{|x-x_0|}{r}\right)^\alpha
\left(\frac{1}{|\Omega_{2r}(x_0)|}\int_{\Omega_{2r}(x_0)}v^2\,dx\right)^{1/2}
\]
for every $x\in \Omega_r(x_0)$.
\end{lm}

Equivalently, if $v$ is harmonic in $\Omega_{2r}(x_0)$ and vanishes on the boundary portion in the Sobolev sense, then
\[
\operatorname{osc}_{\Omega_\rho(x_0)}v
\leq C\left(\frac{\rho}{r}\right)^\alpha
\left(\frac{1}{|\Omega_{2r}(x_0)|}\int_{\Omega_{2r}(x_0)}v^2\,dx\right)^{1/2},
\quad 0<\rho<r.
\]
This is the classical boundary H\"older estimate under the exterior measure density condition; see, for instance, \cite[Theorem 8.29]{DT}.

\subsection{Pointwise norms}

We finally recall the pointwise norms used in the statement of the main results. Let $x_0\in\partial\Omega$, $0<\alpha<1$,  $f$ is a function defined on $\Omega$. We say $f$ is $C^{k,\alpha}$ at $x_0$ in the $L^2$ sense for $k=-2,-1$ if
$$
[f]_{\mathcal{L}^{2,k+\alpha}(x_0)}=:\sup_{0<r<1}\frac 1{r^{k+\alpha}}(\frac 1{|B_r(x_0)|}\int_{\Omega\cap B_r(x_0)}|f(x)|^2dx)^{\frac 12}<\infty.
$$
We say $f$ is $C^{k,\alpha}$ at $x_0$ for $k=0, 1, 2$ if
$$
[f]_{C^{k,\alpha} (x_0)}=: \inf_{P_k}\sup_{x\in \Omega\cap B_1\backslash\{x_0\}}\frac{|f(x)-P_k|}{|x-x_0|^{k+\alpha}}<+\infty.
$$
 We say $f$ is $C^{k,\alpha}$ at $x_0$ in the $L^2$ sense for $k=0,1,2$ if
$$
[f]_{\mathcal{L}^{2,k+\alpha}(x_0)}=:\inf_{P_k}\sup_{0<r<1}\frac 1{r^{k+\alpha}}(\frac 1{|B_r(x_0)|}\int_{\Omega\cap B_r(x_0)}|f(x)-P_k|^2dx)^{\frac 12}<\infty.
$$  Here $P_k$ is taken over the set of $k-th$ order polynomials. It is well known that when $f\in C^{k,\alpha}(0)$ for $k=0,1,2$, we can define the norm of $||f||_{C^{k,\alpha}(x_0)}$. For example, if $f\in C^{2,\alpha}(x_0)$, since $f$ is twice differentiable at $x_0$, we define $||f||_{C^{2,\alpha}(x_0)}=|f(x_0)|+|Df(x_0)|+||D^2f(x_0)||+[f]_{C^{2,\alpha}(x_0)}$. Similarly, we define $||f||_{\mathcal{L}^{2,\alpha}(x_0)}=|f(x_0)|+[f]_{\mathcal{L}^{2,\alpha}(x_0)}$.

\begin{thm}[Campanato Embedding Theorem]
Let $E\subset\mathbb R^n$ satisfy the measure density condition
\[
|E\cap B_r(x)|\geq c_0 r^n
\]
for every $x\in E$ and every $0<r<r_0$. Let $0<\alpha<1$ and $m\in\mathbb N\cup\{0\}$.

Assume that for every $x\in E$ and every $0<r<r_0$, there exists a polynomial $P_{x,r}$ of degree at most $m$ such that
\[
\left(\frac{1}{|E\cap B_r(x)|}\int_{E\cap B_r(x)} |u-P_{x,r}|^2\,dy\right)^{1/2}
\leq M r^{m+\alpha}.
\]
Then $u$ has a representative in $C^{m,\alpha}(E)$, and
\[
\|u\|_{C^{m,\alpha}(E)}
\leq C\left(\|u\|_{L^2(E)}+M\right),
\]
where $C$ depends only on $n,m,\alpha,c_0$ and $r_0$.

\end{thm}
\section{the compactness of L-uniform and $\theta$-admissible domains}

The compactness of L-uniform domains was previously discussed by Li-Wang \cite{LW} and Du-Li-Wang \cite{DLW2023}. In particular, in \cite{DLW2023}, the authors showed the L-uniform domains have uniformly bounded nonlocal perimeters and thus have an $L^1$ limit up to a subsequence by the fractional Sobolev compact embedding theorem. Here we will use the PDE method and  Vitali covering lemma to show the compactness of a sequence of  L-uniform and $\theta$-admissible domains in the Hausdorff distance. Our proposition is following.

\begin{prop}[Compactness of L-uniform domains]\label{lm2}
Let $\{\Omega_k\}\subset B_2$ be a sequence of $L$-uniform and $\theta$-admissible domains. If $\inf\limits_{k}diam(\Omega_k)\geq C_0$
, then $\{\Omega_k\}$ is  compact in the Hausdorff distance.
\end{prop}

\begin{rem}\label{2.1}For $L$-uniform domain $\Omega$, from \cite{DLW2023}, there are two obvious facts. One fact is that for any $x\in \partial \Omega$ and $0<r< diam (\Omega) $ there is a constant $C=C(L, n)>0$ such that there is a ball of radius $Cr$ contained in $B_r(x)\cap \Omega$. The other fact is that if $\Omega$ is an $L$-uniform domain with $diam(\Omega)\geq C_0$, then there is $r=r(L, n, C_0)>0$ such that $\Omega$ contains a ball of radius $r$.
\end{rem}

Before proving the proposition, we show the following lower bound estimate.
\begin{lm}[Lower bound estimate]\label{lmlower}
Let $\Omega\subset B_2$ be an $L$-uniform domain. We assume that
 $diam(\Omega)\geq 4R$, and $\overline{B_{R}(y_0)}\subset \Omega$. If  $u$ satisfies that
\begin{eqnarray*}
\left\{
\begin{array}{rcll}
-\Delta u&=&0\qquad&\text{in}~~\Omega\backslash\overline{B_{R}(y_0)},\\
u&=&1\qquad&\text{on}~~\partial B_{R}(y_0), \\
u&=&0\qquad&\text{on}~~\partial \Omega,
\end{array}
\right.
\end{eqnarray*}
then we have $u(x)\geq C_{1}d^{\beta}(x,\partial\Omega)$ for any $x\in \Omega\backslash\overline{B_{R}(y_0)}$. Here $C_1$ and $\beta$ are the positive constant which are dependent only on $R$, $L$ and $n$.
\end{lm}
\begin{proof}
   We only need to prove the lower bound holds when $x$ near boundary. For simplicity, we denote $(\Omega)_{s}=\{x\in\Omega~|d(x,\partial\Omega)> s\}$.
Let $\delta$  small to be determined later, by Harnack inequality, we have
$$u(x)\geq C(\delta)~,\forall x\in(\Omega)_{\delta}\backslash\overline{B_{R}(y_0)}.$$
We next claim that if $u(y)\geq C_2,~\forall y \in (\Omega)_{2^{-k}\delta}\backslash\overline{B_{R}(y_0)}$, then for any $x\in(\Omega)_{2^{-(k+1)}\delta}\backslash(\Omega)_{2^{-k}\delta}$, $u(x)\geq C_{3}C_2$, where $C_3=C_3(L,R,n)\in (0,1)$ is a universal constant independent of $k$ and $\delta$.

To prove the claim, by the definition of uniform domain, for any $x\in(\Omega)_{2^{-(k+1)}\delta}\backslash(\Omega)_{2^{-k}\delta}$, there exists a rectifiable curve $\gamma$ connects $x$ and $y_0$ with $\gamma (0)=x$ and $\gamma(1)=y_0$, which satisfies (\ref{2}) and (\ref{3}).  Let $T=\min\limits_{t}\{t~|~d(\gamma(t),\partial\Omega)\geq 2^{-k}\delta\}$, then choose $\delta$ such that $L2^{-k}\delta<R-2^{-k}\delta$ (for example, $\delta<\frac{R}{1+L}$ independent of $k$), we deduce that
$$2^{-k}\delta=d(\gamma(T),\partial\Omega)\geq \frac{1}{L}|\gamma(T)-\gamma(0)|,$$
which implies
\begin{eqnarray}\label{uni1}
    |\gamma(T)-\gamma(0)|\leq L2^{-k}\delta.
\end{eqnarray}
Again by the definition of uniform domain, there exists a rectifiable curve $\gamma_2$(may different from $\gamma$) connects $x$ and $\gamma(T)$, with $\gamma _2(0)=x$ and $\gamma_2(1)=\gamma (T)$, which satisfies (\ref{2}) and (\ref{3}). Together with $(\ref{uni1})$, we have
\begin{eqnarray}\label{uni2}
    \mathcal{H}^{1}(\gamma_2)\leq L|\gamma(T)-\gamma(0)|\leq L^{2}2^{-k}\delta.
\end{eqnarray}
Recall that $d(\gamma(T),\partial\Omega)=2^{-k}\delta$, $d(x,\partial\Omega)\geq 2^{-(k+1)}\delta$, by Harnack inequality, there exists $p,q\in \gamma_2$, $d(p, \gamma(T))=\frac{2^{-k}\delta}{2}$, $d(q,x)=\frac{2^{-(k+1)}\delta}{2}$, such that $C_2\leq u(\gamma(T))\leq C(n)u(p)$, $u(q)\leq C(n)u(x)$. Choose $r=\min\{\frac{2^{-(k+1)}\delta}{L},\frac{2^{-(k+2)}\delta}{L}\}$, by the finite length of $\gamma_2$, we can connect $p,q$ by at most $8L^2$ balls of radius $r$. By chain of Harnack balls, we can prove the claim.

For any $x$ near the boundary, there exists $k$, such that $2^{-(k+1)}\delta\leq d=d(x,\partial\Omega)\leq 2^{-k}\delta$. By the claim above,
we have $u(x)\geq C_{3}^{k}C(\delta)$. Together with $k+1\geq -\frac{\ln d-\ln\delta}{ln2}$, it follows that
$$u(x)\geq d^{\frac{-\ln C_{3}}{\ln 2}}(\frac{\delta}{2})^{\frac{\ln C_{3}}{\ln2}}C(\delta).$$
We choose $\beta=-\frac{\ln C_{3}}{\ln 2}$, and $C_1=(\frac{\delta}{2})^{\frac{\ln C_{3}}{\ln2}}C(\delta)$.  Thus we complete the proof of the lemma.
\end{proof}

\
\

{\bf{Proof of Proposition $\ref{lm2}$:}} First, by definition and Remark \ref{2.1}, for any sequence $\{x_k\}$ with $x_k\subset\partial\Omega_k$ and with $x_k\rightarrow x_0$ up to a subsequence,  there exists $\{y_k\}$ such that $y_k\subset \Omega_k$, and further there exists $R>0$ depending only on $L$, $C_0$ and $n$ such that, for any integer $k\geq 1$, $d(\partial\Omega_k,y_k)\geq 4R$ and $y_k\rightarrow y_0$ up to a subsequence and $B_{R}(y_0)\subset \Omega_k$ for all $k$.

 Next, we will construct harmonic functions on $\Omega_k$, and show that the limit of these harmonic functions is also a harmonic function, which induces a limit domain we need. To be specific, we construct harmonic functions $\{u_k\}$ as follow:
\begin{eqnarray*}
\left\{
\begin{array}{rcll}
-\Delta u_k&=&0\qquad&\text{in}~~\Omega_{k}\backslash\overline{B_{R}(y_0)},\\
u_k&=&1\qquad&\text{on}~~\partial B_{R}(y_0), \\
u_k&=&0\qquad&\text{on}~~\partial \Omega_k.
\end{array}
\right.
\end{eqnarray*}
Clearly, there exist a constant $\alpha\in (0,1) $ depending only with $\theta, L, C_0$ and $n$ such that for all $ k\geq 1 $  we have $$u_k\in C^{\alpha}(\overline{\Omega_{k}\backslash\overline{B_{R}(y_0)}}),~\|u_k\|_{C^\alpha(\overline{\Omega_{k}\backslash\overline{B_{R}(y_0)})}} \leq C_1 $$
for some constant $C_1=C_1(\theta,n,R)$. We claim that it holds for some constant $\beta>0$ depending only with $ M, C_0$ and $n$ that
\begin{eqnarray}\label{keyeq}
   C_{2}d^{\beta}(x,\partial\Omega_k)\leq u_k(x)\leq C_{3}d^{\alpha}(x,\partial\Omega_k),~\forall x\in\Omega_{k}\backslash\overline{B_{R}(y_0)},~\forall k\geq1
\end{eqnarray}
for some constant $C_2=C_2(L, n, R)$ and $C_3=C_3(\theta,n,R,\alpha)$. The right side is exactly $C^{\alpha}$ regularity of harmonic funtions up to the boundary. The left side of $(\ref{keyeq})$ holds for Lemma $\ref{lmlower}$.

By properties of harmonic functions, we can extend $u_k$ to $0$ outside $B_{2}\backslash\Omega_{k}$, to $1$ in $B_{R}(y_0)$, which we still denote as $u_k$. Notice that $u_k$ is equicontinuous and uniform bounded, by Arzela-Ascoli lemma, $u_k$ has a  convergent subsequence. For simplicity, we denote the subsequence as $u_k$, the limitation as $u_{\infty}$. By our definition, $\Omega_{k}=\{u_{k}>0\}$. Set $\Omega_\infty=\{u_\infty>0\}$. We claim that
$\Omega_k$ converges to $\Omega_{\infty}$ in the Hausdorff distance. To this purpose, it is sufficient to show, for any $ \varepsilon>0$, there is a positive integer $N=N(\varepsilon)$, such that, for $k>N$, the following properties hold:
\begin{itemize}
    \item[(I)] $\Omega_k\subset (\Omega_{\infty})^{\varepsilon}$,
    \item[(II)] $\Omega_{\infty}\subset (\Omega_{k})^{\varepsilon}$,
    \item[(III)] $(\Omega_{\infty})_{\varepsilon}\subset\Omega_k$,
    \item[(IV)] $(\Omega_k)_{\varepsilon}\subset\Omega_\infty$.
\end{itemize}
Here $(\Omega_\infty)_{\varepsilon}=\{x\in\Omega_\infty~|d(x,\partial\Omega_\infty )> \varepsilon\}$,  $(\Omega_\infty)^{\varepsilon}=\{x\in \R^n~|d(x, \Omega_\infty ) < \varepsilon\}$, and the similar notations of $(\Omega_k)_\varepsilon$ and $(\Omega_k)^\varepsilon$.

For (I), without loss of generality, suppose that $x \in\Omega_k\backslash \overline{B_{R}(y_0)}$ and $\varepsilon >0$ is small, by definition we have $u_{k}(x)>0$. Because $\Omega_k$ is M-uniform, there exists $y_k \in B_{\varepsilon}(x)\cap\Omega_{k}$, where $d(y_k,\partial\Omega_k)\geq C\varepsilon$ for $C=C(n,L)$. From lower bound estimate, we have $u_k(y_k)\geq C_2(C\varepsilon)^{\beta}$. By diagonal technique, $u_{\infty}(y_{\infty})>0$ for $y_k\rightarrow y_\infty$ up to a subsequence. Then we can deduce that $y_{\infty}\in\Omega_{\infty}$, which implies (I).  Similarly, (IV) can be shown.

For (II), suppose not, there exists $x\in\Omega_{\infty}$, such that $B_{\varepsilon}(x)\cap\Omega_{k}=\emptyset$ for some small $\varepsilon>0$. Then by definition, $\forall y\in B_{\varepsilon}(x)$, $u_{k}(y)=0$. Let $k\rightarrow\infty$, we have $u_{\infty}(y)=0$ in $B_{\varepsilon}(x)$, which is a contradiction to $u_\infty(x)>0$ since $x\in\Omega_{\infty}$.

For (III), suppose not, there exists $x\in(\Omega_{\infty})_{\varepsilon}$ such that $x\notin\Omega_k$ for some small $\varepsilon>0$. Then by definition, $u_{k}(x)=0$. Let $k\rightarrow\infty$, we have $u_{\infty}(x)=0$ , which is a contradiction.

Now we show that  $\Omega_{\infty}$ is also an $L$-uniform domain. In fact, $\forall x_1,~x_2\in\Omega_{\infty}$, there exists $\varepsilon>0$ small, for sufficiently large integer $k$, such that $x_1,~x_2\in (\Omega_{\infty})_{\varepsilon}\subset\Omega_k\subset (\Omega_{\infty})^{\varepsilon}$.
Since $\Omega_k$ is an $L$-uniform domain, then there exists a rectifiable curve $\gamma(t)\subset\Omega_k$ connecting $x_{1},~x_2$ satisfies property (\ref{2}) and (\ref{3}). It is sufficient to show $\gamma(t)\subset \Omega_\infty$ to the purpose. From our assumption, $B_{\varepsilon}(x_i)\subset\Omega_{\infty}$ for $i=1,2$. Also for all $ y\in \gamma (t)$ with $d(y,x_{i})\geq\varepsilon$ for $i=1,2$, we have $d(y,\partial\Omega_{k})\geq \frac{\varepsilon}{L}$. This implies that $\gamma(t)\subset(\Omega_k)_{\frac{\varepsilon}{L}}\subset\Omega_{\infty}$.

Finally we prove that $\Omega_{\infty}$ is a $\theta$-admissible domain. Without loss of generality, we assume that $0\in\partial\Omega_\infty$, and prove that $|B_1(0)\cap\Omega^{c}_{\infty}|\geq \theta|B_1(0)|$.

Fixing $\varepsilon>0$ small, for any $ x\in B_1(0)\cap\partial\Omega_{k}$, from Remark \ref{2.1}, there exist
$C=C(L,n)>0$, $C'=C'(L,n)>0$ and $y\in \Omega_k$, such that
$y\in B_{C\varepsilon}(x)$, $B_{C'\varepsilon}(y)\subset B_{C\varepsilon}(x)$, $B_{2C'\varepsilon}(y)\subset \Omega_k$. Then $\cup_{x\in\partial\Omega_k\cap B_1(0)} B_{C\varepsilon}(x)\supset (\partial\Omega_k\cap B_1(0))^{\varepsilon}$. Hence there is a countable sequence $\{B_{C\varepsilon}(x_i)\}$ of disjoint balls such that $\cup_{i}5B_{C\varepsilon}(x_i)\supset (\partial\Omega_k\cap B_1(0))^{\varepsilon}$ via Vitali covering lemma.

We compute the Hausdorff measure of $(\partial\Omega_k\cap B_1(0))^\varepsilon $ first. Set $\delta_0=2\beta-2>0$ since $\beta$ in Lemma \ref{lmlower} can be chosen $\beta>1$. Then we have
\begin{eqnarray*}
    \sum_{i}\omega_n 5^{n-\delta_0}C^{n-\delta_0} \varepsilon^{n-\delta_0}
      &\leq& C\sum_{i}\frac{\varepsilon^{n-\delta_0}} {|B_{C'\varepsilon}(y_i)|}\int_{{B_{C'\varepsilon}(y_i)}} 1dx\\
    &\leq& C\sum_{i}\frac{\varepsilon^{n-\delta_0}}{\varepsilon^{n}}\varepsilon^{-2\beta} \int_{{B_{C'\varepsilon}(y_i)}} \varepsilon^{2\beta}dx\\
    &\leq & C\sum_{i}\varepsilon^{-2\beta-\delta_0}\int_{{B_{C'\varepsilon}(y_i)}} u_{k}^{2}dx\\
    &\leq& C\sum_{i}\varepsilon^{-2\beta-\delta_0}\int_{{B_{C\varepsilon}(x_i)}\cap\Omega_k} u_{k}^{2}dx\\
     &\leq& C\sum_{i}\varepsilon^{2-2\beta-\delta_0}\int_{{B_{C\varepsilon}(x_i)}\cap \Omega_k} |\nabla u_{k}|^{2}dx\\
     &\leq & C \int_{\Omega_k}|\nabla u_k|^2dx\\
     &\leq & C,
\end{eqnarray*}
since   $\int_{\Omega_k}|\nabla u_k|^2dx$ is uniformly finite.  Also, a modification of above calculation implies that
$$ \mathcal{H}^{n}((\partial\Omega_k\cap B_1(0))^\varepsilon)\leq C\varepsilon^{\delta_0}.$$
Hence by the face (I)-(IIII), we have for all small $\varepsilon>0$ and the sufficiently large $k$
\begin{eqnarray*}
  |\Omega_{\infty}\cap B_1(0)|&=& |(\Omega_{k})_\varepsilon\cap B_1(0)|+ |(\Omega_{\infty}\backslash (\Omega_k)_\varepsilon)\cap B_1(0)|\\
&\leq & |\Omega_{k}\cap B_1(0)|+|(\Omega_{k})^\varepsilon\backslash (\Omega_k)_\varepsilon\cap B_1(0)|\\
&\leq & |\Omega_{k}\cap B_1(0)|+|(\partial\Omega_{k}\cap B_1(0))^\varepsilon|\\
&\leq & (1-\theta)|B_1(0)|+C\varepsilon^{\delta_0}.
\end{eqnarray*}
Let $\varepsilon\rightarrow 0$, we have $$ |\Omega_{\infty}\cap B_1(0)|\leq (1-\theta)|B_1(0)|,$$
i.e. $$|\Omega^{c}_{\infty}\cap B_1(0)|\geq \theta|B_1(0)|.$$

\begin{rem}\label{HD}Let $\Omega$ be an $L$-uniform and $\theta$-admissible bounded domain. If $diam(\Omega)\geq C_0$, then from the proof of above proposition there exist constants $\delta_0>0$ and $C>0$ such that
$$ \mathcal{H}^{n}((\partial\Omega)^\varepsilon)\leq C\varepsilon^{\delta_0}$$ for all small $\varepsilon>0$. \end{rem}


\qed

\section{boundary pointwise $C^{\alpha}$ estimate}
In this section, we assume that $\Omega$, $f$ and $g$ satisfy the conditions in Theorem \ref{thm1}. Without loss of generality, we assume that $diam(\Omega)>1$. We will prove Theorem \ref{thm1} by using compactness and perturbation techniques. An iterative scheme in the scaling sense  is also involved. We begin with the following energy estimate.
\begin{lm}[Energy estimate]\label{lm1} If $u$ is a weak solution of $(\ref{1})$ in $\Omega$, then for any $\eta \in C_{0}^{\infty}\left(B_{1}\right)$, we have\\
\begin{equation}\label{v1}
\int_{\Omega_1} \eta^{2}|\nabla[(u-\sup\limits_{\partial_w\Omega_1}g)^+]|^{2} \leq \int_{\Omega_1} \eta^{2} {f}^{2}+C \int_{\Omega_1}\left(|\nabla \eta|^{2}+\eta^{2}\right)|u|^{2}+C\|g\|_{L^{\infty}(\partial_w\Omega_1)}^{2},
\end{equation}
\begin{equation}\label{g1}
\int_{\Omega_1} \eta^{2}|\nabla[(u-\inf\limits_{\partial_w\Omega_1}g)^-]|^{2} \leq \int_{\Omega_1} \eta^{2} {f}^{2}+C \int_{\Omega_1}\left(|\
\nabla\eta|^{2}+\eta^{2}\right)|u|^{2}+C\|g\|_{L^{\infty}(\partial_w\Omega_1)}^{2}.
\end{equation}
\end{lm}
\begin{proof}
    Take $\eta^2(u-\sup\limits_{\Omega_1}g)^{+}$ as the test function, we have
    $$\int_{B_1}\nabla [(u-\sup\limits_{\partial_w\Omega_1}g)^{+}]\cdot \nabla[\eta^2(u-\sup\limits_{\partial_w\Omega_1}g)^{+}]dx\leq\int_{B_1}\eta^2(u-\sup\limits_{\partial_w\Omega_1}g)^{+}f^+dx,$$
    i.e.,
    $$\int_{B_1}\eta^2|\nabla[(u-\sup\limits_{\partial_w\Omega_1}g)^{+}]|^2dx\leq\int_{B_1}\eta^2(u-\sup\limits_{\partial_w\Omega_1}g)^{+}f^+dx-2\int_{B_1}\eta \nabla\eta \cdot (u-\sup\limits_{\partial_w\Omega_1}g)^{+}
    \nabla [(u-\sup\limits_{\partial_w\Omega_1}g)^{+}].$$
    Then by using the Cauchy inequality, we obtain
    \begin{eqnarray*}
\int_{B_1}\eta^2|\nabla[(u-\sup\limits_{\partial_w\Omega_1}g)^{+}]|^2dx&\leq&\int_{B_1}\eta^2(u-\sup\limits_{\partial_w\Omega_1}g)^{+}f^+dx-2\int_{B_1}\eta \nabla\eta \cdot (u-\sup\limits_{\partial_w\Omega_1}g)^{+}
    \nabla [(u-\sup\limits_{\partial_w\Omega_1}g)^{+}]\\
    &\leq&\frac{1}{2}\int_{B_1}\eta^2(f^+)^2dx+\frac{1}{2}\int_{B_1}\eta^2[(u-\sup\limits_{\partial_w\Omega_1}g)^{+}]^2dx+\frac{1}{2}\int_{B_1}\eta^2|\nabla[(u-\sup\limits_{\partial_w\Omega_1}g)^{+}]|^2dx\\
&+& 2\int_{B_1}|\nabla\eta|^2[(u-\sup\limits_{\partial_w\Omega_1}g)^{+}]^2dx.
\end{eqnarray*}
Notice that $(u-\sup\limits_{\partial_w\Omega_1}g)^{+}$ vanishes outside $\Omega_1$ and $\int_{B_1}\eta^2(f^+))^2dx\leq\int_{\Omega_1}\eta^2{f}^2dx$, it follows that
$$\int_{\Omega_1} \eta^{2}|\nabla[(u-\sup\limits_{\partial_w\Omega_1}g)^+]|^{2} \leq \int_{\Omega_1} \eta^{2} {f}^{2}+C \int_{\Omega_1}\left(|\nabla \eta|^{2}+\eta^{2}\right)|u|^{2}+C\|g\|_{L^{\infty}(\partial_w\Omega_1)}^{2}.$$
The proof for the second inequality is similar.
\end{proof}

Next we would like to show the linear property of weak solutions with respect to the harmonic function, which is essential in the iterative scheme.
\begin{lm}{\label{lm11}} If $u(x)$ is a weak solution of $(\ref{1})$ in $\Omega$, $h(x)$ is harmonic in  $\R^n$, then $u-h$ is a weak solution in the sense of the definition $\ref{uni}$ to the following problem
\begin{equation*}
\left\{
\begin{array}{rcll}
-\Delta (u-h)&=&f\qquad&\text{in}~~\Omega,\\
u-h&=&g-h\qquad&\text{on}~~\partial\Omega. \\
\end{array}
\right.
\end{equation*}
\end{lm}
\begin{proof} To show $u-h$ is a weak solution in the sense of the definition \ref{uni}, we would like to use the following classical result (also see Lemma 4.6 in \cite{HL2011}), that
if the function $u\in H^1_{loc}(B_r)$  satisfies that
$$
\int_{B_r} \nabla u \nabla \varphi dx\leq \int_{B_r} f^+\varphi dx
$$
for all $\varphi \in C^{\infty}_0(B_r)$ and $\varphi\geq 0$, then $(u-k)^+$ for any constant $k$ is also in $ H^1_{loc}(B_r)$ and  satisfies
\begin{equation}\label{338}
\int_{B_r} \nabla (u-k)^+\nabla \varphi dx\leq \int_{B_r} f^+\varphi dx
\end{equation}for all $\varphi \in C^{\infty}_0(B_r)$ and $\varphi\geq 0$.

Now we show $u-h$ is a weak solution in the sense of the definition \ref{uni}. Firstly, since $u(x)$ is a weak solution of $(\ref{1})$ in $\Omega$, by the definition \ref{uni} of weak solutions, we have
\begin{equation*}
\displaystyle\int_{\Omega}\nabla u\cdot\nabla \varphi dx=\displaystyle\int_{\Omega}f\cdot\varphi dx,\quad \forall~\varphi\in C_{0}^{\infty}(\Omega).
\end{equation*}
Due to $h(x)$ is harmonic, we have
\begin{equation}\label{333}
\displaystyle\int_{\Omega}\nabla (u-h)\cdot\nabla \varphi dx=\displaystyle\int_{\Omega}f\cdot\varphi dx,\quad \forall~\varphi\in C_{0}^{\infty}(\Omega).
\end{equation}

Next, we take any $ 0<r\leq diam(\Omega)$ and $x_0\in\partial\Omega$, and set $\Omega_r=B_r(x_0)\cap \Omega$. We want to show that  $[u-h-\sup\limits_{\partial_\omega \Omega_r}(g-h)]^+$'s 0-extension in $\Omega^c \cap B_r(x_0) $ is in $H^1(B_r(x_0))$. Without loss of generality, we still denote  $[u-h-\sup\limits_{\partial_\omega \Omega_r}(g-h)]^+$'s 0-extension in $\Omega^c \cap B_r(x_0)$  as  $[u-h-\sup\limits_{\partial_\omega \Omega_r}(g-h)]^+$.

 To this purpose, by the continuity of $h$, for any small $\varepsilon>0$, we choose the covering $\{U_i\}_{i=0}^{m}$ of $\Omega_r$, which $U_0\subset\subset \Omega_r$  is an domain, $U_i=B_{r_i}(x_i)$ with $x_i\in \partial_w \Omega_r$ and small $r_i>0$ for $i=1,2,...,m$ such that
$$
|h(x)-h(y)|\leq \frac {\varepsilon}{2},~~~ \forall x,y \in B_{r_i}(x_i).
$$
Let $\{\eta_i\}_{i=0}^{m}$ be a smooth partition of unity subordinate to the covering $\{U_i\}_{i=0}^{m}$, that is, suppose
\begin{eqnarray*}
\left\{
\begin{array}{rl}
0\leq \eta_i\leq 1\qquad&\eta_i\in C^{\infty}_0(U_i),\\
\sum\limits_{i=0}^{m}\eta_i=1\qquad&\text{in}~~\Omega_r.
\end{array}
\right.
\end{eqnarray*}
For any given test function $\varphi (x)\in C^1_0(B_r(x_0))$, since $\text { supp} \eta_0\subset\subset \Omega_r$, we know that
\begin{eqnarray}\label{3333}
&& \int_{B_r(x_0)}(u-h-\sup_{\partial_\omega\Omega_r}(g-h)-\varepsilon)^+D(\varphi\eta_0)dx\nonumber\\
&=& -\int_{\Omega_r\cap\{u>h+\sup\limits_{\partial_\omega\Omega_r}(g-h)+\varepsilon\}}\eta_0\varphi D(u-h-\sup_{\partial_\omega\Omega_r}(g-h)-\varepsilon)dx.
\end{eqnarray}
For $i=1,2,...,m$, we know $\text { supp} \eta_i\subset\subset U_i$. In this case, for any $x\in B_r(x_0)$,  if $[u(x)-h(x)-\sup\limits_{\partial_\omega\Omega_r}(g-h)-\varepsilon]^+>0$, it holds that $x\in \Omega_r$. Furthermore, when $x\in \Omega_{r_i}=B_{r_i}(x_i)\cap \Omega$ and $[u(x)-h(x)-\sup\limits_{\partial_\omega\Omega_r}(g-h)-\varepsilon]^+>0$, we have
$$[u(x)-h(x)-\sup\limits_{\partial_\omega\Omega_r}(g-h)-\varepsilon]^+
=[[u(x)-\sup\limits_{\partial_\omega\Omega_{r_i}}g]^+ +
\sup\limits_{\partial_\omega\Omega_{r_i}}g-h(x)-\sup\limits_{\partial_\omega\Omega_r}(g-h)-\varepsilon]^+,$$
since $\sup\limits_{\partial_\omega\Omega_{r_i}}g-h(x)-\sup\limits_{\partial_\omega\Omega_r}(g-h)\leq \varepsilon$ in $\Omega_{r_i}$, and
\begin{eqnarray*}
u(x) & > & h(x)+ \sup\limits_{\partial_\omega\Omega_r}(g-h)+\varepsilon\\
&\geq  & h(x)+ \sup\limits_{\partial_\omega\Omega_{r_i}}(g-h)+\varepsilon\\
&\geq & h(x)+ \sup\limits_{\partial_\omega\Omega_{r_i}}g-\sup\limits_{\partial_\omega\Omega_{r_i}}h+\varepsilon\\
&\geq & \sup\limits_{\partial_\omega\Omega_{r_i}}g-\frac{\varepsilon}2+\varepsilon\\
&\geq & \sup\limits_{\partial_\omega\Omega_{r_i}}g+\frac \varepsilon 2.
\end{eqnarray*}

On the other hand, when $x\in \Omega_{r_i}$ and $[[u(x)-\sup\limits_{\partial_\omega\Omega_{r_i}}g]^+ +
\sup\limits_{\partial_\omega\Omega_{r_i}}g-h(x)-\sup\limits_{\partial_\omega\Omega_r}(g-h)-\varepsilon]^+>0$, we also have
$$[u(x)-h(x)-\sup\limits_{\partial_\omega\Omega_r}(g-h)-\varepsilon]^+
=[[u(x)-\sup\limits_{\partial_\omega\Omega_{r_i}}g]^+ +
\sup\limits_{\partial_\omega\Omega_{r_i}}g-h(x)-\sup\limits_{\partial_\omega\Omega_r}(g-h)-\varepsilon]^+.$$
Thus we obtain that in each ball $B_{r_i}(x_i)$ it holds
\begin{equation}\label{337}[u(x)-h(x)-\sup\limits_{\partial_\omega\Omega_r}(g-h)-\varepsilon]^+
=[[u(x)-\sup\limits_{\partial_\omega\Omega_{r_i}}g]^+ +
\sup\limits_{\partial_\omega\Omega_{r_i}}g-h(x)-\sup\limits_{\partial_\omega\Omega_r}(g-h)-\varepsilon]^+.
\end{equation}
Hence since the zero extension in $\Omega^c\cap B_{r_i}(x_i)$ of $[u-\sup\limits_{\partial_\omega \Omega_{r_i}}g]^+$ is in $H^1(B_{r_i}(x_i))$,  we have $[[u(x)-\sup\limits_{\partial_\omega\Omega_{r_i}}g]^+ +
\sup\limits_{\partial_\omega\Omega_{r_i}}g-h(x)-\sup\limits_{\partial_\omega\Omega_r}(g-h)-\varepsilon]^+$ is in $H^1(B_{r_i}(x_i))$. Consequently by (\ref{337}) we have  the zero extension in $\Omega^c\cap B_{r_i}(x_i)$ of $ [u(x)-h(x)-\sup\limits_{\partial_\omega\Omega_r}(g-h)-\varepsilon]^+$ is in $H^1(B_{r_i}(x_i))$, and
\begin{eqnarray}\label{3334}
&& \int_{B_r(x_0)}(u-h-\sup_{\partial_\omega\Omega_r}(g-h)-\varepsilon)^+D(\varphi\eta_i)dx \nonumber\\
&=& -\int_{\Omega_r\cap\{u>h+\sup\limits_{\partial_\omega\Omega_r}(g-h)+\varepsilon\}}\eta_i\varphi D(u-h-\sup_{\partial_\omega\Omega_r}(g-h)-\varepsilon) dx
\end{eqnarray}
Putting (\ref{3333}) and (\ref{3334}) together, by using $\sum\limits_{i=0}^{m}\eta_i=1$, we obtain that  the zero extension in $\Omega^c\cap B_{r}(x_0)$ of $ [u(x)-h(x)-\sup\limits_{\partial_\omega\Omega_r}(g-h)-\varepsilon]^+$ is in $H^1(B_{r}(x_0))$, and
\begin{eqnarray}\label{3335}
&& \int_{B_r(x_0)}(u-h-\sup_{\partial_\omega\Omega_r}(g-h)-\varepsilon)^+D\varphi dx\nonumber\\
&=& - \int_{\Omega_r\cap\{u>h+\sup\limits_{\partial_\omega\Omega_r}(g-h)+\varepsilon\}}\varphi D(u-h-\sup_{\partial_\omega\Omega_r}(g-h)-\varepsilon)dx.
\end{eqnarray}

We also notice that, for above smooth partition of unity $\eta_i\in C^{\infty}_0(U_i)$ with $i=0,1,2,\cdots, m$, similarly as the proof of classical result (\ref{338}), we get from (\ref{333}) that
\begin{equation}\label{334}
\displaystyle\int_{B_r(x_0)}\nabla [u-h-\sup\limits_{\partial_\omega \Omega_r}(g-h)-\varepsilon]^+\cdot\nabla (\eta_0\varphi) dx\leq\displaystyle\int_{B_r(x_0)}f^+\eta_0\varphi dx,\quad \forall~\varphi\in C_{0}^{\infty}(B_r(x_0)), \varphi \geq 0.
\end{equation}
Since $u(x)$ is a weak solution of (\ref{1}), we have for each $i\in \{1,2,\cdots, m\}$,
\begin{equation*}
\displaystyle\int_{B_r(x_0)}\nabla [u-\sup\limits_{\partial_\omega \Omega_{r_i}}g]^+\cdot\nabla (\eta_i\varphi) dx\leq\displaystyle\int_{B_r(x_0)}f^+\eta_i\varphi dx,\quad \forall~\varphi\in C_{0}^{\infty}(B_r(x_0)), \varphi \geq 0.
\end{equation*}
Hence by using that $h$ is harmonic to get
\begin{equation*}
\displaystyle\int_{B_r(x_0)}\nabla ([u-\sup\limits_{\partial_\omega \Omega_{r_i}}g]^+-h(x))\cdot\nabla (\eta_i\varphi) dx\leq\displaystyle\int_{B_r(x_0)}f^+\eta_i\varphi dx,\quad \forall~\varphi\in C_{0}^{\infty}(B_r(x_0)), \varphi \geq 0.
\end{equation*}
Similarly as the proof of the classical result (\ref{338}) again, we obtain that
\begin{equation*}
\displaystyle\int_{B_r(x_0)}\nabla [[u(x)-\sup\limits_{\partial_\omega\Omega_{r_i}}g]^+ +
\sup\limits_{\partial_\omega\Omega_{r_i}}g-h(x)-\sup\limits_{\partial_\omega\Omega_r}(g-h)-\varepsilon]^+\cdot\nabla (\eta_i\varphi) dx\leq\displaystyle\int_{B_r(x_0)}f^+\eta_i\varphi dx,
\end{equation*}
for any $\varphi\in C_{0}^{\infty}(B_r(x_0))$ and $ \varphi \geq 0$. Thus we have from (\ref{337})
\begin{eqnarray}\label{335}
& & \displaystyle\int_{B_r(x_0)}\nabla [u-h-\sup\limits_{\partial_\omega \Omega_r}(g-h)-\varepsilon]^+\cdot\nabla (\eta_i\varphi) dx\nonumber\\
&=& \displaystyle\int_{B_r(x_0)}\nabla [[u(x)-\sup\limits_{\partial_\omega\Omega_{r_i}}g]^+ +
\sup\limits_{\partial_\omega\Omega_{r_i}}g-h(x)-\sup\limits_{\partial_\omega\Omega_r}(g-h)-\varepsilon]^+\cdot\nabla (\eta_i\varphi) dx\nonumber\\
&\leq & \displaystyle\int_{B_r(x_0)}f^+\eta_i\varphi dx
\end{eqnarray}
for any $\varphi\in C_0^{\infty}(B_r(x_0))$ with $\varphi\geq 0$.  Putting (\ref{334}) and (\ref{335}) together, by using $\sum\limits_{i=0}^{m}\eta_i=1$ again, it follows that
 \begin{equation}\label{3336}
\displaystyle\int_{B_r(x_0)}\nabla [u-h-\sup\limits_{\partial_\omega \Omega_r}(g-h)-\varepsilon]^+\cdot\nabla \varphi dx\leq\displaystyle\int_{B_r(x_0)}f^+ \varphi dx,\quad \forall~\varphi\in C_{0}^{\infty}(B_r(x_0)), \varphi \geq 0.
\end{equation}
Since we have proved that the zero extension in $\Omega^c\cap B_{r}(x_0)$ of $ [u(x)-h(x)-\sup\limits_{\partial_\omega\Omega_r}(g-h)-\varepsilon]^+$ is in $H^1(B_{r}(x_0))$, we can take $\varphi=\eta[u(x)-h(x)-\sup\limits_{\partial_\omega\Omega_r}(g-h)-\varepsilon]^+$ in (\ref{3336}) for some cut off function $\eta$, and know from the energy inequality that  the zero extension in $\Omega^c\cap B_{r}(x_0)$ of $ [u(x)-h(x)-\sup\limits_{\partial_\omega\Omega_r}(g-h)-\varepsilon]^+$ is uniformly bounded  in $H^1_{loc}(B_{r}(x_0))$.  Without loss of generality, we assume that  the zero extension in $\Omega^c\cap B_{r}(x_0)$ of $ [u(x)-h(x)-\sup\limits_{\partial_\omega\Omega_r}(g-h)-\varepsilon]^+$ is uniformly bounded  in $H^1(B_{r}(x_0))$. Otherwise, we can discuss the problem on a small ball $B_{r-\delta}(x_0)$.
Thus, we have  the zero extension in $\Omega^c\cap B_{r}(x_0)$ of $
[u(x)-h(x)-\sup\limits_{\partial_\omega\Omega_r}(g-h)-\varepsilon]^+ $ converges weakly to  $
[u(x)-h(x)-\sup\limits_{\partial_\omega\Omega_r}(g-h)]^+ $ in $H^1(B_{r}(x_0)) $ up to the subsequence as $\varepsilon\to 0^+$. Hence  the zero extension in $\Omega^c\cap B_{r}(x_0)$ of $[u(x)-h(x)-\sup\limits_{\partial_\omega\Omega_r}(g-h)]^+ $ in $H^1(B_{r}(x_0)) $. Furthermore, taking $\varepsilon \to 0^+$ in (\ref{3335}) and (\ref{3336}) up to the subsequence, we obtain
$$
\int_{B_r(x_0)}(u-h-\sup_{\partial_\omega\Omega_r}(g-h))^+D\varphi dx=-
\int_{\Omega_r\cap\{u>h+\sup\limits_{\partial_\omega\Omega_r}(g-h)\}}\varphi D(u-h-\sup_{\partial_\omega\Omega_r}(g-h))dx.
$$
and
 \begin{eqnarray}\label{336}
 \displaystyle\int_{B_r(x_0)}\nabla [u-h-\sup\limits_{\partial_\omega \Omega_r}(g-h)]^+\cdot\nabla \varphi dx\leq  \displaystyle\int_{B_r(x_0)}f^+\varphi dx, \quad \forall~\varphi\in C_{0}^{\infty}(B_r(x_0)) \text { and } \varphi\geq 0.
\end{eqnarray}

Similarly, we also can show that  the zero extension in $\Omega^c\cap B_r(x_0)$ of $[u-h-\inf\limits_{\partial_\omega \Omega_r}(g-h)]^-$ is in $H^1(B_r(x_0))$, and
\begin{eqnarray*}
 \displaystyle\int_{B_r(x_0)}\nabla [u-h-\inf\limits_{\partial_\omega \Omega_r}(g-h)]^-\cdot\nabla \varphi dx\leq  \displaystyle\int_{B_r(x_0)}f^-\varphi dx
\end{eqnarray*} for any $\varphi\in C_0^{\infty}(B_r(x_0))$ with $\varphi\geq 0$. Thus by the definition \ref{uni} we know that then $u-h$ is a weak solution of
\begin{equation*}
\left\{
\begin{array}{rcll}
-\Delta (u-h)&=&f\qquad&\text{in}~~\Omega,\\
u-h&=&g-h\qquad&\text{on}~~\partial\Omega. \\
\end{array}
\right.
\end{equation*}

\end{proof}

Next we show the following approximation lemma by the compactness method.
\begin{lm}[Compactness Lemma]\label{lm3}
  For any $\varepsilon>0$, there exists a small $\delta=\delta(\varepsilon)>0$ such that for any weak solution of $(\ref{1})$ in $\Omega$ with
  $$\displaystyle\frac{1}{|B_{1}|}\displaystyle\int_{\Omega_{1}}u^{2}dx\leq 1,~~ \frac 1{|B_1|}\int_{\Omega_{1}}f^{2}dx\leq \delta^2, ~~ \|g\|_{L^{\infty}(\partial_w\Omega_1)}\leq \delta,$$  there exists a $L$-uniform and $\theta$-admissible domain $\tilde{\Omega}$ with $|\Omega_{\frac 12}\Delta \tilde{\Omega}_{\frac 12}|\leq \varepsilon$ and $0\in \partial \tilde {\Omega}$, and a function $h(x)$ defined in $\tilde{\Omega}_\frac 12$, which is a solution of
\begin{eqnarray*}
\left\{
\begin{array}{rcll}
-\Delta h&=&0\qquad&\text{in}~~\tilde{\Omega}_{\frac 12},\\
h&=&0\qquad&\text{on}~~\partial_{w}\tilde {\Omega}_{\frac 12}, \\
\end{array}
\right.
\end{eqnarray*}
with
$$\int_{\tilde{\Omega}_{\frac 12}}|h|^{2}dx\leq2 |B_1|,
$$
such that
$$\int_{\tilde{\Omega}\cap \Omega_{\frac{1}{2}}}|u-h|^{2}dx\leq\varepsilon^{2}.
$$
\end{lm}
\begin{proof}
We prove it by contradiction. If not, there is an $\varepsilon_0 >0$ so that for each integer k, there is a weak solution of
\begin{eqnarray*}
\left\{
\begin{array}{rcll}
-\Delta u_k&=&f_k\qquad&\text{in}~~\Omega_{k},\\
u_k&=&g_k\qquad&\text{on}~~\partial\Omega_{k}, \\
\end{array}
\right.
\end{eqnarray*}
with
$$\displaystyle\frac{1}{|B_{1}|}\displaystyle\int_{\Omega_{k}\cap B_1}u_{k}^{2}dx\leq 1, \frac {1}{|B_1|}\int_{\Omega_{k}\cap B_1}f_{k}^{2}dx\leq \frac{1}{k^2},  \|g_k\|_{L^{\infty}(\partial\Omega_k\cap B_1)}\leq \frac{1}{k},$$
and $\Omega_k$ is a sequence of $L$-uniform and $\theta$-admissible domains with $0\in \partial \Omega_k $ and $C\geq diam(\Omega_k)\geq 1$,
such that  for any bounded $L$-uniform and $\theta$-admissible domains  $\tilde{\Omega}$ with $|(\Omega_k\cap B_{\frac 12})\Delta \tilde{\Omega}_{\frac 12}|\leq \varepsilon_0$ and $0\in \partial \tilde {\Omega}$, and for any function $h$  which is  the solution of
\begin{eqnarray*}
\left\{
\begin{array}{rcll}
-\Delta h&=&0\qquad&\text{in}~~\tilde{\Omega}_{\frac 12},\\
h&=&0\qquad&\text{on}~~\partial_{\omega}\tilde{\Omega}_{\frac 12}, \\
\end{array}
\right.
\end{eqnarray*}
with $$\int_{\tilde{\Omega}_{\frac 12}}|h|^{2}dx\leq2 |B_1|,
$$
the following inequality holds,
\begin{equation}\label{a}\int_{\tilde{\Omega}\cap\Omega_{k}\cap B_{\frac{1}{2}}}|u_{k}-h|^{2}dx\geq\varepsilon_{0}^{2}.
\end{equation}
By Proposition $\ref{lm2}$, we obtain that $\{\Omega_k\}$ has a subsequence , which we still denoted as $\Omega_k$,  that converges to $\Omega_{\infty}$ in the Hausdorff distance. Here $\Omega_\infty$ is a bounded $L$-uniform and $\theta$-admissible domain with $diam(\Omega_\infty)\geq 1$,  $0\in \partial \Omega_{\infty}$. Furthermore, there exists an positive integer $K=K(\varepsilon_0)$ such that for all $k>K$ it holds $|(\Omega_{\infty}\Delta\Omega_{k})\cap B_{\frac{1}{2}}|\leq \varepsilon_0$.

From lemma $\ref{lm1}$ we observe that $$\|(u_k-\sup\limits_{\partial\Omega_k\cap B_1}g_k)^{+}\|_{H^1(B_{\frac{3}{4}})}\leq C,$$
$$\|(u_k-\inf\limits_{\partial\Omega_k\cap B_1}g_k)^{-}\|_{H^1(B_{\frac{3}{4}})}\leq C.$$
Consequently, using the diagonal subsequence technique, it is easy to see that $(u_k-\sup\limits_{\partial\Omega_k\cap B_1}g_k)^{+}$ [$(u_k-\inf\limits_{\partial\Omega_k\cap B_1}g_k)^{-}$, resp.] has a subsequence, which we still denoted as $(u_k-\sup\limits_{\partial\Omega_k\cap B_1}g_k)^{+}$ [$(u_k-\inf\limits_{\partial\Omega_k\cap B_1}g_k)^{-}$, resp.] that the weak limit $v_1$[$v_2$, resp.] $\in H^{1}(B_{\frac{3}{4}})$ and strongly converges to $v_1$[$v_2$,  resp.] in $L^{2}(B_{\frac{3}{4}})$.  However, we can only  deduce that $u_k$ weakly converges to $u_{\infty}$ in $L^{2}(B_{\frac{3}{4}})$. Our aim is to prove that $u_k$ converges to $u_{\infty}$ strongly in $L^{2}(B_{\frac{3}{4}})$ and $u_{\infty}^{+}=v_1 $, $u_{\infty}^{-}=v_2 $ a.e.. Then we have $u_\infty=v_1-v_2$ a.e..

We write $u_k$ as $u_k^{+}-u_k^{-}$, and let $u_k=0$ in $((\Omega_k)^c)^o$. Notice that $(u_k-\sup\limits_{\partial\Omega_k\cap B_1}g_k)^{+}$ is a Cauchy sequence in $L^{2}(B_{\frac{3}{4}})$, so is $u_{k}^{+}$. In fact, we have
\begin{eqnarray*}\|u^{+}_{k+t}-u^{+}_{k}\|_{L^{2}(B_{\frac{3}{4}})}
&\leq& \|u^{+}_{k+t}-(u_{k+t}-\sup\limits_{\partial\Omega_{k+t}\cap B_1}g_{k+t})^{+}\|_{L^{2}(B_{\frac{3}{4}})}\\
& & +\|(u_{k+t}-\sup\limits_{\partial\Omega_{k+t}\cap B_{1}}g_{k+t})^{+}-(u_{k}-\sup\limits_{\partial\Omega_k\cap B_1}g_{k})^{+}\|_{L^{2}(B_{\frac{3}{4}})}
\\
& & +\|(u_{k}-\sup\limits_{\partial\Omega_k\cap B_1}g_{k})^{+}-u^{+}_{k}\|_{L^{2}(B_{\frac{3}{4}})} \\
&\rightarrow& 0\quad \text{as} \quad k,t\rightarrow \infty
\end{eqnarray*}
From above we can assume that $u_k^+$ converges to $u_{\infty}^+$ strongly in $L^{2}(B_{\frac{3}{4}})$. Since
$$|u^{+}_{\infty}-v_1|
\leq |u^{+}_{\infty}-u^{+}_{k}|+| u^+_{k}- (u_k-\sup\limits_{\partial\Omega_{k}\cap B_1}g_{k})^{+}|+| (u_k-\sup\limits_{\partial\Omega_{k}\cap B_1}g_{k})^{+}-v_1|,
$$
We have $u_{\infty}^{+}=v_1 $ a.e.. Similarly we have
$u_k^-$ converges to $u_{\infty}^-$ strongly in $L^{2}(B_{\frac{3}{4}})$ and $u_{\infty}^{-}=v_2 $ a.e. .

By definition, $v_1$, $v_2$ vanishes outside $\Omega_\infty^c\cap B_{\frac{3}{4}}$. So $u_{\infty}=0$ on $\partial\Omega_{\infty}\cap B_{\frac{1}{2}}$. While for $\forall x_0\in \Omega_{\infty}\cap B_{\frac{1}{2}}$, there exists an integer $k_0$, such that $x_0\in \Omega_k$ for all $k\geq k_0$  and $B_{\varepsilon}(x_0)\subset \Omega_k\cap B_{\frac 12}$ for some $\varepsilon=\varepsilon(d(x_0,\partial \Omega_\infty ))>0$. Then for any $\eta \in C_{0}^{\infty}(B_{\varepsilon_{0}}(x_0))$, since that $u_k$ is a weak solution in $\Omega_k$, we have
$$\int_{B_{\varepsilon_{0}}(x_0)}f_k\eta dx=\int_{B_{\varepsilon_{0}}(x_0)} \nabla u_k\cdot \nabla\eta dx.$$
By  the energy estimate, we know $Du_k$ converges weakly to $Du_\infty$ in $H^1_{loc}(\Omega_\infty \cap B_{\frac 12})$. Hence, if letting $k\rightarrow\infty$, we have $u_{\infty}$ is a weak solution of
\begin{eqnarray*}
\left\{
\begin{array}{rcll}
-\Delta h&=&0\qquad&\text{in}~~\Omega_{\infty}\cap B_{\frac{1}{2}},\\
h&=&0\qquad&\text{on}~~\partial\Omega_{\infty}\cap B_{\frac{1}{2}}. \\
\end{array}
\right.
\end{eqnarray*}
It is clear that
$$\int_{\Omega_{\infty}\cap B_{\frac 12}}|u_\infty|^{2}dx\leq2 |B_1|.
$$
Thus we yield a harmonic function $u_{\infty}$ in $\Omega_{\infty}\cap B_{\frac{1}{2}}$, which is contradict to (\ref{a}) if we take $\tilde {\Omega}=\Omega_\infty$ and $h=u_\infty$.
\end{proof}

In the following we give a key lemma, which will be used repeatedly later.
\begin{lm}[Key Lemma]\label{key2} There exists $0<\lambda<1,~\delta_{0}>0,~0<\alpha_0<1$ such that for any weak solution of $(\ref{1})$
with
$$\displaystyle{\frac{1}{|B_1|}}\int_{\Omega_{1}}u^{2}dx\leq 1, ~\frac 1{|B_1|}\int_{\Omega_{1}}f^{2}dx\leq \delta_{0}^2,~\|g\|_{L^{\infty}(\partial_w\Omega_1)}\leq \delta_0,$$
it follows that for any $0<\alpha <\alpha_0$
$$\left(\frac{1}{|B_{\lambda}|}\int_{\Omega_{\lambda}}|u|^{2}dx\right)^{\frac{1}{2}}\leq\lambda^{\alpha}.
$$
\end{lm}
\begin{proof}
Let $\tilde{\Omega}$ and $h$ be the domain and the function of the previous lemma which satisfies that  $|\Omega_{\frac 12}\Delta \tilde{\Omega}_{\frac 12}|\leq \varepsilon$, $0\in \partial \tilde {\Omega}$,
$$\int_{\tilde{\Omega}_{\frac 12}}|h|^{2}dx\leq2|B_1|,
$$
and
$$
\int_{\tilde{\Omega}\cap\Omega_{\frac{1}{2}}}|u-h|^{2}dx\leq \varepsilon^2
$$
for some $\varepsilon<1$ to be determined. By the property of harmonic functions, it follows that
$$\|h\|_{L^{\infty}(B_{\frac{1}{4}}\cap\tilde{\Omega})}\leq C\|h\|_{L^{2}(B_{\frac{1}{2}}\cap\tilde{\Omega})}\leq C.
$$
Since $\tilde{\Omega}$ is $\theta$-admissible, there exist constants $0<r_{0}\leq\displaystyle\frac{1}{4}$, $0<\alpha_{0}<1$ and $C_{0}=C_{0}(n)$ such that for any $x\in B_{r_{0}}\cap\tilde{\Omega}$,
$$|h(x)|=|h(x)-h(0)|\leq C\frac{|x|^{\alpha_0}}{r_{0}^{\alpha_0}}\|h\|_{L^{\infty}(B_{\frac{1}{4}}\cap\Omega)}\leq C_{0}\frac{|x|^{\alpha_0}}{r_{0}^{\alpha_0}}.
$$
Notice that for each $0<\lambda<r_0$ and for small $\delta_0>0$ to be determined
\begin{eqnarray*}
\int_{\Omega_{\lambda}\backslash \tilde{\Omega}}|u|^{2}dx
&=&\int_{(\Omega_{\lambda}\backslash \tilde{\Omega})\cap\{u>\delta_0\}}|u|^{2}dx+\int_{(\Omega_{\lambda}\backslash \tilde{\Omega})\cap\{-\delta_0<u<\delta_0\}}|u|^{2}dx+\int_{(\Omega_{\lambda}\backslash \tilde{\Omega})\cap\{u<-\delta_0\}}|u|^{2}dx\\
&\leq & \int_{(\Omega_{\lambda}\backslash \tilde{\Omega})\cap\{u>\sup\limits_{\partial_w\Omega_1}g\}}|u|^{2}dx+\int_{(\Omega_{\lambda}\backslash \tilde{\Omega})\cap\{-\delta_0<u<\delta_0\}}|u|^{2}dx+\int_{(\Omega_{\lambda}\backslash \tilde{\Omega})\cap\{u<\inf\limits_{\partial_w\Omega_1}g\}}|u|^{2}dx
\end{eqnarray*}
By the definition of weak solutions, we know that  $(u-\sup\limits_{\partial_w\Omega_1}g)^{+}$ and $(u-\inf\limits_{\partial_w\Omega_1}g)^{-}$ are in $ H^{1}(B_1)$ after they take $0$-extensions in $\Omega^c\cap B_1$. Applying H\"older inequality and Sobolev embedding to $(u-\sup\limits_{\partial_w\Omega_1}g)^+$, and using the energy estimate, we have
\begin{eqnarray*}
& & \frac 1{|B_\lambda|}\int_{(\Omega_{\lambda}\backslash \tilde{\Omega})\cap\{u>\sup\limits_{\partial_w\Omega_1}g\}}|u|^{2}dx\\
&\leq &\frac 2{|B_\lambda|}\int_{(\Omega_{\lambda}\backslash \tilde{\Omega})\cap\{u>\sup\limits_{\partial_w\Omega_1}g\}}|(u-\sup\limits_{\partial_w\Omega_1}g)^{+}|^{2}dx+\frac 2{|B_\lambda|}\int_{(\Omega_{\lambda}\backslash \tilde{\Omega})\cap\{u>\sup\limits_{\partial_w\Omega_1}g\}}\delta_0^{2}dx
\\
&\leq& \frac {2|\Omega_{\lambda}\backslash \tilde{\Omega}|^{\frac{2}{n}}} {|B_\lambda|}\|(u-\sup\limits_{\partial_w\Omega_1}g)^{+}\|^{2}_{L^{2*}(\Omega_{\frac 12})}+2\delta_0^2\\
&\leq& \frac {2\varepsilon^{\frac 2n}}{|B_\lambda|}\|(u-\sup\limits_{\partial_w\Omega_1}g)^+\|^{2}_{H^1(B_{\frac 12} )}+2\delta_0^2\\
&\leq& \frac {C\varepsilon^{\frac 2n}}{|B_\lambda|}+2\delta_0^2 .
\end{eqnarray*}
Similarly, we also have
$$
\frac 1{|B_\lambda|}\int_{(\Omega_{\lambda}\backslash \tilde{\Omega})\cap\{u<\inf\limits_{\partial_w\Omega_1}g\}}|u|^{2}dx \leq \frac {C\varepsilon^{\frac 2n}}{|B_\lambda|}+2\delta_0^2.
$$
Therefor we have
$$
\frac 1{|B_\lambda|}\int_{\Omega_{\lambda}\backslash \tilde{\Omega}}|u|^{2}dx\leq \frac {C\varepsilon^{\frac 2n}}{|B_\lambda|}+5\delta_0^2.
$$
Further more, for each $0<\lambda<r_{0}$, we have
\begin{eqnarray*}
\frac{1}{|B_{\lambda}|}\int_{\Omega_{\lambda}}|u|^{2}dx &=& \frac{1}{|B_{\lambda}|}\int_{\Omega_{\lambda}\cap \tilde{\Omega}}|u|^{2}dx+\frac 1{|B_\lambda|}\int_{\Omega_{\lambda}\backslash \tilde{\Omega}}|u|^{2}dx\\
&\leq &\frac{1}{|B_{\lambda}|}\int_{\Omega_{\lambda}\cap \tilde{\Omega}}|u-h(0)|^{2}dx+\frac {C\varepsilon^{\frac 2n}}{|B_\lambda|}+5\delta_0^2\\
&\leq&\frac{2}{|B_{\lambda}|}\int_{\Omega_{\lambda}\cap \tilde{\Omega}}|u-h|^{2}dx
+\frac{2}{|B_{\lambda}|}\int_{\Omega_{\lambda}\cap \tilde{\Omega}}|h-h(0)|^{2}dx +\frac {C\varepsilon^{\frac 2n}}{|B_\lambda|}+5\delta_0^2\\
&\leq&\frac{C\varepsilon^{\frac {2}{n}}}{|B_{\lambda}|}+2C_{0}^{2}\frac{\lambda^{2\alpha_0}}{r_{0}^{2\alpha_0}}+5\delta_0^2.
\end{eqnarray*}
Now for any $0<\alpha<\alpha_0$, we take $\lambda$ small enough such that
$$2C_{0}^{2}\frac{\lambda^{2\alpha_0}}{r_{0}^{2\alpha_0}}\leq\frac{1}{3}\lambda^{2\alpha},
$$
and further we take $\varepsilon$ small sufficiently such that
$$\frac{C\varepsilon^{\frac 2n}}{|B_{\lambda}|}\leq\frac{1}{3}\lambda^{2\alpha}
$$
and  $\delta_{0}=\min\{\frac {\lambda^{\alpha}}{\sqrt{15}}, \delta(\varepsilon)\}$, where $\delta(\varepsilon)$ is in Lemma $\ref{lm3}$. Thus  Lemma follows.
\end{proof}

{\bf{Proof of Theorem $\ref{thm1}$:}}

We first discuss the normalization of the estimates. Without loss of generality, we may assume $g(0)=0$. Otherwise, let $v(x)=u(x)-g(0)$. Then the estimate of $v$ will translate to that for $u$.

Let $\delta_0$ be as in Lemma $\ref{key2}$. We can assume the normalization conditions:
$$ \displaystyle{\frac{1}{|B_1|}}\int_{\Omega_{1}}u^{2}dx\leq 1,~
 [f]^{2}_{\mathcal{L}^{2,-2+\alpha}(0)}\leq\delta_{0}^{2}, ~  [g]_{C^{\alpha}(0)}\leq \delta_0. $$
Otherwise, we set
$$v(x)=\displaystyle\frac{u(x)}{\left(\displaystyle\frac{1}{|B_{1}|}\int_{\Omega_{1}}u^{2}dx
\right)^{\frac{1}{2}}+\displaystyle\frac{1}{\delta_{0}}\left([f]_{\mathcal{L}^{2,-2+\alpha}(0)}+[g]_{C^{\alpha}(0)}\right)}
\triangleq\displaystyle\frac{u(x)}{N},\quad x\in B_{1}\cap\Omega.
$$
Then $v(x)$ is a weak solution of
\begin{equation}\label{4-2}
\left\{
\begin{array}{rcll}
-\Delta v&=&{f}_{N}\qquad&\text{in}~~\cap\Omega,\\
v&=&g_{N}\qquad&\text{on}~~\partial\Omega,\\
\end{array}
\right.
\end{equation}
where ${f}_{N}=\displaystyle\frac{f}{N}$, $g_{N}=\displaystyle\frac{g}{N}$. Moreover, it follows that
\begin{equation*}\label{4-3}\left(\frac{1}{|B_{1}|}\displaystyle\int_{\Omega_{1}}v^{2}dx\right)^{\frac{1}{2}}=
\frac{\left(\displaystyle\frac{1}{|B_{1}|}\displaystyle\int_{\Omega_{1}}u^{2}dx\right)^{\frac{1}{2}}}
{N}\leq1,
\end{equation*}
and
$$
[f_{N}]^{2}_{\mathcal{L}^{2,-2+\alpha}(0)}\leq\delta_{0}^{2}, ~~~ [g_N]_{C^{\alpha}(0)}\leq \delta_0.
$$

Now, we prove the following claim inductively:
\begin{equation}\label{ite3}
\left(\frac{1}{|B_{\lambda^{k}}|}\int_{\Omega_{\lambda^{k}}}u^{2}dx\right)^{\frac{1}{2}}\leq \lambda^{k\alpha},\quad \forall k\geq0.
\end{equation}
When $k=0$, (\ref{ite3}) is the condition on $u$. By normalized assumption, we know that
$$
\frac 1{|B_1|}\int_{\Omega_1}|f(x)|^2dx\leq [f]^2_{\mathcal{L}^{2,-2+\alpha}(0)}\leq \delta_0^2,
$$
and
$$
\|g\|_{L^{\infty}(\partial_w\Omega_1)}\leq [g]_{C^{\alpha}(0)}\leq\delta_0.
$$
Hence $(\ref{ite3})$  holds for $k=1$  from Lemma $\ref{key2}$.
Now let us assume  that the conclusion is true for $k>1$. Let
$$w(x)=\frac{u(\lambda^{k}x)}{\lambda^{k\alpha}},\quad \tilde{\Omega}=\{x:\lambda^{k}x\in \Omega\}.
$$
It is easy to check that for any $k\geq 0$, $\tilde{\Omega}$ is $L$-uniform and $\theta$-admissible,
and $w$ is a weak solution to
\begin{eqnarray*}
\left\{
\begin{array}{rcll}
-\Delta w&=&\tilde{f}\qquad&\text{in}~~\tilde{\Omega},\\
w&=&\tilde{g}\qquad&\text{on}~~\partial\tilde{\Omega}, \\
\end{array}
\right.
\end{eqnarray*}
where $\tilde{f}(x)
=\displaystyle\frac{f(\lambda^{k}x)}{\lambda^{k(\alpha-2)}}$, $\tilde{g}(x)
=\displaystyle\frac{g(\lambda^{k}x)}{\lambda^{k\alpha}}$.
Thus by using the inductive assumption, we obtain that
$$\frac{1}{|B_{1}|}\displaystyle\int_{\tilde{\Omega}_{1}}w^{2}dx=
\frac{1}{|B_{1}|}\displaystyle\int_{\tilde{\Omega}_{1}}\frac{|u(\lambda^{k}x)|^{2}}{\lambda^{2k\alpha}}dx
=\frac{\displaystyle\frac{1}{|B_{\lambda^{k}}|}\displaystyle\int_{\Omega_{\lambda^{k}}}|u|^{2}dx}{\lambda^{2k\alpha}}
\leq1,
$$
$$\frac 1{|B_1|}\int_{\tilde{\Omega}_{1}}\tilde{f}^{2}dx=\frac 1{|B_1|}
\int_{\Omega_{1}}\frac{|f(\lambda^{k}x)|^{2}}{\lambda^{2k(\alpha-2)}}dx
=\frac{\frac 1{|B_{\lambda^k}|}\displaystyle\int_{\Omega_{\lambda^k}}|f|^{2}dx}{\lambda^{2k(\alpha-2)}}
\leq[f]^{2}_{\mathcal{L}^{2,-2+\alpha}(0)}\leq\delta_{0}^{2},
$$
and
$$\|\tilde{g}\|_{L^{\infty}(\partial_w\tilde{\Omega}_1)}=\|\frac{g(\lambda^{k}x)}{\lambda^{k\alpha}}\|_{L^{\infty}(\partial_w\tilde{\Omega}_1)} \leq [g]_{C^\alpha(0)}\leq \delta_{0}.$$
Hence we can apply Lemma $\ref{key2}$ for $w$ to obtain that
$$\left(\frac{1}{|B_{\lambda}|}\int_{\tilde{\Omega}_{\lambda}}w^{2}dx\right)^{\frac{1}{2}}\leq\lambda^{\alpha}.
$$
We scale back to get
$$\left(\frac{1}{|B_{\lambda^{k+1}}|}\int_{{\Omega}_{\lambda^{k+1}}}|u(x)|^{2}dx\right)^{\frac{1}{2}}
\leq\lambda^{(k+1)\alpha} .
$$
Thus we prove the $(k+1)$-th step.

Finally, for $0<r\leq 1$, there is a k with $\lambda^{k+1}\leq r\leq \lambda^{k}$, and that
\begin{eqnarray*}
\frac{1}{|B_{r}|}\int_{{\Omega}_{r}}|u(x)|^{2}dx&\leq&\frac{1}{\lambda^{n}|B_{\lambda^{k}}|}\int_{{\Omega}_{\lambda^{k}}}|u(x)|^{2}dx\\
&\leq& \frac{\lambda^{2k\alpha}}{\lambda^{n}}\\
&\leq& \frac {1}{\lambda^{n+2\alpha }}r^{2\alpha}.
\end{eqnarray*}
This implies that $u$ is $C^\alpha$ at 0 in the $L^{2}$ sense, and
$$
[u]_{\mathcal{L}^{2,\alpha}(0)}\leq C\left(\displaystyle\frac{1}{|B_{1}|}\int_{\Omega_{1}}u^{2}dx
\right)^{\frac{1}{2}}+C\left([f]_{\mathcal{L}^{2,-2+\alpha}(0)}+||g||_{C^{\alpha}(0)}\right)
$$
for some universal constant $C$. Thus we complete the proof of Theorem \ref{thm1}.
\
\

\section{Boundary pointwise $C^{1,\alpha}$ estimate}
In this section, we will prove the boundary pointwise $C^{1,\alpha}$ estimate  at $0\in \partial \Omega$ which is stated in Theorem \ref{thm2}. We would like to point out that by the assumption of $\partial\Omega\in C^{1,\alpha}(0)$, it can be inferred that $$|x_n|\leq [\partial \Omega]_{C^{1,\alpha}(0)}|x'|^{1+\alpha}$$ for any $x\in \partial _\omega\Omega_1 $. We denote that $K_\Omega=[\partial \Omega]_{C^{1,\alpha}(0)}$.

First we show the following approximation lemma by the compactness method.
\begin{lm}[Compactness lemma]\label{lm4}
There exists a small $\delta=\delta(\varepsilon)>0$ such that for any weak solution of $(\ref{1})$ in $\Omega$ with $$\displaystyle\frac{1}{|B_{1}|}\displaystyle\int_{\Omega_{1}}u^{2}dx\leq 1, ~\frac 1{|B_1|}\int_{\Omega_{1}}f^{2}dx\leq \delta^2, ~\|g\|_{L^{\infty}(\partial_w\Omega_1)}\leq \delta, ~K_\Omega\leq \delta,$$
there exists a harmonic function $h(x)$ defined in $B_{\frac{1}{2}}$ with
$$\int_{B_{\frac 12}}|h(x)|^2dx\leq 2 |B_1|$$
and $h$ is odd in $x_n$
such that
$$\int_{B_{\frac{1}{2}}\cap\Omega}|u-h|^{2}dx\leq\varepsilon^{2}.
$$
\end{lm}
\begin{proof} First we notice that $diam(\Omega_1)\geq 2-\delta\geq 1$ if $\delta$ is small.
We prove it by contradiction. If not, there is an $\varepsilon_0 >0$ so that for each integer $k$, there is a weak solution of
\begin{eqnarray*}
\left\{
\begin{array}{rcll}
-\Delta u_k&=&f_k\qquad&\text{in}~~\Omega_{k},\\
u_k&=&g_k\qquad&\text{on}~~\partial\Omega_{k}, \\
\end{array}
\right.
\end{eqnarray*}
with
$$\displaystyle\frac{1}{|B_{1}|}\displaystyle\int_{\Omega_{k}\cap B_1}u_{k}^{2}dx\leq 1, ~\frac{1}{|B_1|}\int_{\Omega_{k}\cap B_1}f_{k}^{2}dx\leq \frac{1}{k^2}, ~\|g_k\|_{L^{\infty}(\partial\Omega_k\cap B_1)}\leq \frac{1}{k},~K_{\Omega_{k}}\leq \frac{1}{k},$$
such that  for any harmonic function $h$ in $B_{\frac{1}{2}}$, which is odd in $x_n$,
with
$$\int_{B_{\frac 12}}|h(x)|^2dx\leq 2 |B_1|,$$
 the following inequality holds
\begin{equation}\label{c}\int_{B_{\frac{1}{2}}\cap\Omega_k}|u_{k}-h|^{2}dx\geq\varepsilon_{0}^{2}.
\end{equation}

From lemma $\ref{lm1}$ we observe that $$\|(u_k-\sup\limits_{\partial\Omega_k\cap B_1}g_k)^{+}\|_{H^1(B_{\frac{3}{4}})}\leq C,$$
$$\|(u_k-\inf\limits_{\partial\Omega_k\cap B_1}g_k)^{-}\|_{H^1(B_{\frac{3}{4}})}\leq C.$$

Similar to lemma $\ref{lm3}$, we know that $(u_k-\sup\limits_{\partial\Omega_k\cap B_1}g_k)^{+}$ and $(u_k-\inf\limits_{\partial\Omega_k\cap B_1}g_k)^{-}$ converge  weakly to $v_1$ and $v_2$ respectively in $ H^{1}(B_{\frac{3}{4}})$, and converge strongly to $v_1$ and $v_2$ respectively in $L^{2}(B_{\frac{3}{4}})$ up to subsequence. If we let $u_k=0$ in $((\Omega_k)^c)^o$, we also know that $u_k$ converges strongly to $v$ in $L^{2}(B_{\frac{3}{4}})$ for $v=v_1-v_2$. It also holds that
$$\int_{B_{\frac 12}}|v(x)|^2dx\leq 2 |B_1|.$$

For any $ x_0\in  B^{+}_{\frac{1}{2}}$, since $K_{\Omega_k}\leq \frac 1k$, there exists $k_0$ and $t_0>0$, such that for any $k\geq k_0$, $x_0\in \Omega_k$ and $B_{t_0}(x_0)\subset \Omega_k$. Then for any $\eta \in C_{0}^{\infty}(B_{t_{0}}(x_0))$,
$$\int_{B_{t_{0}}(x_0)}f_k\eta dx=\int_{B_{t_{0}}(x_0)} Du_k\cdot D\eta dx.$$
Letting $k\rightarrow\infty$, we have $v$ is a harmonic function in $B_{t_0}(x_0)$. Furthermore, we have $v$ is harmonic function in $B^+_{\frac 12}$.

From \cite{BW04},  it is clear that $v=0$ on $T_{\frac 12}$ since $K_{\Omega_k}\leq \frac 1k$. Set $h$ is the odd extension of $v|_{B^+_{\frac 12}}$. Then $h$ is harmonic in $B_{\frac{1}{2}}$, which is contradict to (\ref{c}).
\end{proof}

In the following we give a key lemma, which will be used repeatedly later.
\begin{lm}[Key Lemma]\label{key22} There is a $C>0$ such that for any $0<\alpha<1$, there exists $0<\lambda<1,~\delta_{0}>0$ such that for any weak solution of $(\ref{1})$
with
$$\displaystyle{\frac{1}{|B_1|}}\int_{\Omega_{1}}u^{2}dx\leq 1,~~
\frac 1{|B_1|}\int_{\Omega_{1}}f^{2}dx\leq \delta_{0}^2,  ~~ \|g\|_{L^{\infty}(\partial_w\Omega_1)}\leq \delta_0,~~ K_\Omega\leq \delta_0,$$ it follows that
$$\left(\frac{1}{|B_{\lambda}|}\int_{\Omega_{\lambda}}|u-p|^{2}dx\right)^{\frac{1}{2}}\leq\lambda^{1+\alpha},
$$
where $p(x)=ax_n,$ with $|a|\leq C.$
\end{lm}
\begin{proof}
Let $h$ be the function of the previous lemma which satisfies that
$$
\int_{\Omega\cap B_{\frac{1}{2}}}|u-h|^{2}dx\leq \varepsilon^2
$$
for some $\varepsilon<1$ to be determined, and
$$\int_{B_{\frac{1}{2}}}|h|^2\leq2 |B_1|.$$
By the property of harmonic functions, for $x\in B_{\frac{1}{4}}$,
$$|D^{k}h(x)|\leq \displaystyle{\frac{C}{|B_1|}}\int_{B_{\frac{1}{2}}}h^{2}dx\leq C
$$ for $k=1, 2$.
Now let $p(x)$ be the first order Taylor polynomial of $h$ at $0$. Since $h=0$ on $T_{\frac{1}{2}}$, we have that the form of the polynomial $p(x)$ is
$$p(x)=ax_n,$$
where $|a|\leq |Dh(0)|\leq C$.  It also holds that for any $x\in B_{\frac{1}{4}}$,
$$|(h-p)(x)|\leq C |x|^2.
$$
Therefore for each $0<\lambda<\frac{1}{4}$, we have
\begin{eqnarray*}
\frac{1}{|B_{\lambda}|}\int_{\Omega_{\lambda}}|u-p|^{2}dx
&\leq&\frac{2}{|B_{\lambda}|}\int_{\Omega_{\lambda}}|u-h|^{2}dx
+\frac{2}{|B_{\lambda}|}\int_{\Omega_{\lambda}}|h-p|^{2}dx\\
&\leq&\frac{2\varepsilon^{2}}{|B_{\lambda}|}+2C^{2}\lambda^{4}.
\end{eqnarray*}
Now for any $0<\alpha<1$, we take $\lambda$ small enough such that
$$2C^{2}\lambda^{4}\leq\frac{1}{2}\lambda^{2(1+\alpha)},
$$
and further we take $\varepsilon$ small sufficiently such that
$$\frac{2\varepsilon^{2}}{|B_{\lambda}|}\leq\frac{1}{2}\lambda^{2(1+\alpha)}
$$
and  $\delta_{0}=\delta(\varepsilon)$ in Lemma $\ref{lm4}$. Thus  Lemma follows.
\end{proof}

{\bf{Proof of Theorem $\ref{thm2}$:}}

Without loss of generality we assume that $g(0)=0$ and $Dg(0)=0$. Otherwise, we may consider $v(x)=u(x)-g(0)-Dg(0)\cdot x$. Let $\delta_0$ be as in Lemma $\ref{key22}$. We can assume the normalization conditions:
$$ \displaystyle{\frac{1}{|B_1|}}\int_{\Omega_{1}}u^{2}dx\leq 1,~
 [f]^{2}_{\mathcal{L}^{2,-1+\alpha}(0)}\leq\delta_{0}^{2}, ~  [g]_{C^{1,\alpha}(0)}\leq \frac{\delta_0}{2},~ [\partial\Omega]_{C^{1,\alpha}(0)}\leq \frac{\delta_0}{C_0}, $$
where $C_0$ is a constant to be determined later. Otherwise, we set
$$v(y)=\displaystyle\frac{u(x)}{\left(\displaystyle\frac{1}{|B_{1}|}\int_{\Omega_{1}}u^{2}dx
\right)^{\frac{1}{2}}+\displaystyle\frac{1}{\delta_{0}}\left([f]_{\mathcal{L}^{2,-1+\alpha}(0)}+2[g]_{C^{1,\alpha}(0)}\right)}
\triangleq\displaystyle\frac{u(x)}{N},
$$ where $y=\frac x{R}$ and $x\in B_{1}\cap\Omega$. By choosing $R$ small enough which depends only on $n$ and $K_\Omega$, the normalization conditions hold. Without loss of generality, we assume that $R=1$.

By definition, we have
$$|x_n|\leq [\partial\Omega]_{C^{1,\alpha}(0)}|x'|^{1+\alpha}\leq \frac{\delta_0}{C_0} |x'|^{1+\alpha},~~\forall x\in \partial\Omega\cap B_1.$$
Now, we prove the following claim inductively: there are harmonic polynomials
$$p_k(x)=a^{k}x_n$$
such that
\begin{equation}\label{ite2}
\left(\frac{1}{|B_{\lambda^{k}}|}\int_{\Omega_{\lambda^{k}}}|u-p_k|^{2}dx\right)^{\frac{1}{2}}\leq \lambda^{k(1+\alpha)},\quad \forall k\geq 0.
\end{equation}
and $a^0=0$,
\begin{align}\label{co2}
    |a^{k}-a^{k-1}|&\leq C \lambda^{\alpha (k-1)}, \quad \forall k\geq 1.
    \end{align}
By the normalized assumption, we know that $k=0$ is just the assumption and the claim holds for $k=1$  since Lemma $\ref{key22}$.
Now let us assume  that the conclusion is true for $k$. Let
$$w(x)=\frac{(u-p_k)(\lambda^{k}x)}{\lambda^{k(1+\alpha)}},\quad \tilde{\Omega}=\{x:\lambda^{k}x\in \Omega\}.
$$
It is easy to check that for any $k\geq 0$, $\tilde{\Omega}$ is $C^{1,\alpha}$ at $0$,
and by Lemma \ref{lm11} $w$ is a weak solution to
\begin{eqnarray*}
\left\{
\begin{array}{rcll}
-\Delta w&=&\tilde{f}\qquad&\text{in}~~\tilde{\Omega},\\
w&=&\tilde{g}\qquad&\text{on}~~\partial\tilde{\Omega}, \\
\end{array}
\right.
\end{eqnarray*}
where $\tilde{f}(x)
=\displaystyle\frac{f(\lambda^{k}x)}{\lambda^{k(\alpha-1)}}$, $\tilde{g}(x)
=\displaystyle\frac{(g-p_k)(\lambda^{k}x)}{\lambda^{k(1+\alpha)}}$.
Thus by using the inductive assumption, we obtain that
$$\frac{1}{|B_{1}|}\displaystyle\int_{\tilde{\Omega}_{1}}w^{2}dx=
\frac{1}{|B_{1}|}\displaystyle\int_{\tilde{\Omega}_{1}}\frac{|u-p_k)(\lambda^{k}x)|^{2}}{\lambda^{2k(\alpha+1)}}dx
=\frac{\displaystyle\frac{1}{|B_{\lambda^{k}}|}\displaystyle\int_{\Omega_{\lambda^{k}}}|u-p_k|^{2}dx}{\lambda^{2k(\alpha+1)}}
\leq1,
$$
$$\frac 1{|B_1|}\int_{\tilde{\Omega}_{1}}\tilde{f}^{2}dx=\frac 1{|B_1|}
\int_{\Omega_{1}}\frac{|f(\lambda^{k}x)|^{2}}{\lambda^{2k(-1+\alpha)}}dx=\frac 1{\lambda^{2k(-1+\alpha)}|B_{\lambda^k}|}
\int_{\Omega_{\lambda^k}}|f(y)|^{2}dy
\leq [f]^{2}_{\mathcal{L}^{2.-1+\alpha}(0)}
\leq\delta_{0}^{2},
$$
$$K_{\tilde{\Omega}}= \lambda^{k\alpha}K_{\Omega}\leq \frac{\delta_0}{C_0},$$
\begin{eqnarray*}
\|\tilde{g}\|_{L^{\infty}(\partial\tilde\Omega\cap B_1)}&=&\|\frac{(g-p_k)(y)}{\lambda^{k(1+\alpha)}}\|_{L^{\infty}(\partial\Omega\cap B_{\lambda^{k}})}\\
&\leq& \|\frac{g(y)}{\lambda^{k(1+\alpha)}}\|_{L^{\infty}(\partial\Omega\cap B_{\lambda^{k}})}+\|\frac{a^{k}y_n}{\lambda^{k(1+\alpha)}}\|_{L^{\infty}(\partial\Omega\cap B_{\lambda^{k}})}\\
&\leq & [g]_{C^{1,\alpha}(0)}+CK_{\Omega}\\
&\leq & \frac {\delta_0}{2}+\frac {C}{C_0}\delta_0\leq\delta_{0}.
\end{eqnarray*} Here we choose $C_0$ such that $\frac{C}{C_0}\leq \frac 12$. Now we can apply Lemma $\ref{key22}$ for $w$ to obtain that
$$\left(\frac{1}{|B_{\lambda}|}\int_{\tilde{\Omega}_{\lambda}}|w-\bar{p}|^{2}dx\right)^{\frac{1}{2}}\leq\lambda^{1+\alpha}
$$ where $\bar{p}(x)=ax_n$ for $x\in \tilde{\Omega}_\lambda$
with $|a|\leq C$. We scale back to get
$$\left(\frac{1}{|B_{\lambda^{k+1}}|}\int_{{\Omega}_{\lambda^{k+1}}}|u(x)-p_{k}(x)-\lambda^{(1+\alpha)k}\bar{p}(\frac{x}{\lambda^k})|^{2}dx\right)^{\frac{1}{2}}
\leq\lambda^{(k+1)(1+\alpha)} .
$$
Thus we prove the $(k+1)$-th step by choosing $p_{k+1}(x)=p_{k}(x)+\lambda^{(1+\alpha)k}\bar{p}(\frac{x}{\lambda^k})$. It is clear that
$$|a^{k+1}-a^k|\leq \lambda^{\alpha k}|a|\leq C\lambda^{\alpha k}.$$
It is clear that ${a^k}$ is Cauchy sequence and converges to $a_\infty$. The limiting polynomial is $$ p_\infty(x)=a_\infty x_n$$ satisfies
$$|p_k(x)-p_\infty(x)|\leq C\lambda^{k(1+\alpha)},~~~\text { for any } |x|\leq \lambda^k. $$
Hence,  for any $r\leq 1$, there is a $k$ with $\lambda^{k+1}\leq r\leq \lambda^{k}$, and that
\begin{eqnarray*}
\frac{1}{|B_{r}|}\int_{{\Omega}_{r}}|u(x)-p_\infty|^{2}dx&\leq&\frac{1}{\lambda^{n}|B_{\lambda^{k}}|}\int_{{\Omega}_{\lambda^{k}}}|u(x)-p_{k}(x)|^{2}dx+\frac{1}{\lambda^{n}|B_{\lambda^{k}}|}\int_{{\Omega}_{\lambda^{k}}}|p_\infty(x)-p_{k}(x)|^{2}dx\\
&\leq& \frac{\lambda^{2k(1+\alpha)}}{\lambda^{n}}+\frac{C\lambda^{2k(1+\alpha)}}{\lambda^{n}}\\
&\leq& Cr^{2(1+\alpha)}
.
\end{eqnarray*}
This implies that $u$ is $C^{1,\alpha}$ at 0 in the $L^{2}$ sense, and
$$
[u]_{\mathcal{L}^{2,1+\alpha}(0)}\leq C\left(\displaystyle\frac{1}{|B_{1}|}\int_{\Omega_{1}}u^{2}dx
\right)^{\frac{1}{2}}+C\left([f]_{\mathcal{L}^{2,-1+\alpha}(0)}+||g||_{C^{1,\alpha}(0)}\right).
$$
Thus we complete the proof of Theorem \ref{thm2}.

\section{Boundary pointwise $C^{2,\alpha}$ estimate}
Different from boundary pointwise $C^{1,\alpha}$ regularity, the main difficulty for boundary pointwise $C^{2,\alpha}$ regularity lies in that we do not have the following estimate:
$$|x_n|\leq C|x'|^{2+\alpha}, ~\forall x\in \partial_\omega \Omega_1.$$
As a result, we cannot check the condition in Lemma $\ref{key3}$ when we iterate. In the case of non-divergence form equation, in \cite{LZ2020}, the authors show the pointwise $C^{2,\alpha}$ estimats under the assumption of $Du(0)=0$, then use the normalization technique to remove the assumption $Du(0)=0$. Let $K_\Omega=[\partial \Omega]_{C^{2,\alpha}(0)}$, and $S$ is a symmetric matrix. Without loss of generality, we assume that
$$
B_{1}\cap \{(x',x_n)|x_n>\frac 12 x'^TSx'+K_\Omega|x'|^{2+\alpha}\}\subset B_{1}\cap \Omega,
  $$
and
$$
B_{1}\cap \{(x',x_n)|x_n<\frac 12 x'^TSx'-K_\Omega|x'|^{2+\alpha}\}\subset B_1\cap \Omega^{c}.
$$

As the previous section, we have the following compactness lemma.

\begin{lm}[Compactness lemma]\label{lm5}
For any $\varepsilon>0$, there exists a small $\delta=\delta(\varepsilon)>0$ such that for any weak solution of $(\ref{1})$ in $\Omega$ with
$$\displaystyle\frac{1}{|B_{1}|}\displaystyle\int_{\Omega_{1}}u^{2}dx\leq 1,~~\frac{1}{|B_1|}\int_{\Omega_{1}}f^{2}dx\leq \delta^2,~~\|g\|_{L^{\infty}(\partial_w\Omega_1)}\leq \delta, ~~ ||S||\leq \delta,~~K_\Omega\leq \delta,$$
there exists a harmonic function $h(x)$ in $B_{\frac 12}$ and $h$ is odd in $x_n$ such that
$$\int_{B_{\frac{1}{2}}}|h|^{2}dx\leq2|B_1|$$
and
$$\int_{\Omega\cap B_{\frac{1}{2}}}|u-h|^{2}dx\leq\varepsilon^{2}.$$
\end{lm}

\begin{lm}[Key Lemma]\label{key3} There is $C>0$ such that for any $0<\alpha<1$, there exists $0<\lambda<1,~\delta_{0}>0$ such that for any weak solution of $(\ref{1})$
with
$$\displaystyle{\frac{1}{|B_1|}}\int_{\Omega_{1}}u^{2}dx\leq 1,~~\frac 1{|B_1|}\int_{\Omega_{1}}f^{2}dx\leq \delta_{0}^2, ~~\|g\|_{L^{\infty}(\partial_w\Omega_1)}\leq \delta_0,~~K_\Omega\leq \delta_0,~~||S||\leq \delta_0,$$
it follows that
$$\left(\frac{1}{|B_{\lambda}|}\int_{\Omega_{\lambda}}|u-p|^{2}dx\right)^{\frac{1}{2}}\leq\lambda^{2+\alpha},
$$
where $p(x)=ax_{n}-\frac {a}{2} x'^TSx'+b_{in}x_{n}x_{i}+\frac{a}{2} tr(S)x_{n}^{2}$ is harmonic with
$$|a|+|b_{in}|\leq C,$$
for $i=1,2,\cdots, n-1$. Here the Einstein summation convention is used. \end{lm}
\begin{proof}
Let $h$ be the harmonic function of the previous lemma which satisfies that
$$
\int_{\Omega\cap B_{\frac{1}{2}}}|u-h|^{2}dx\leq \varepsilon^2
$$
for some $\varepsilon<1$ to be determined, and
$$\int_{B_{\frac{1}{2}}}|h|^2\leq2|B_1|.$$
By the property of harmonic functions, for $x\in B_{\frac{1}{4}},$  it follows that
$$\|h\|_{L^{\infty}(B_{\frac{1}{4}})}\leq C\|h\|_{L^{2}(B_{\frac{1}{2}})}\leq C.
$$
Hence we have  for $k=1, 2, 3$
$$|D^{k}h(x)|\leq C, \text{ for } x\in B_{\frac{1}{8}}.
$$
Now let $\tilde{p}(x)$ be the second order Taylor polynomial of $h$ at $0$,  as $h$ is odd at $x_n$ direction, we have that the form of the polynomial is
$$ \tilde{p}(x)=ax_{n}+b_{in}x_{n}x_{i}.$$  Here $i=1,2,\cdots, n-1$. We know that $\tilde{p}(x)$ is  also harmonic.

It is clear that
$$|a|+|b_{in}|\leq C, \quad \text { for } i=1,2,\cdots n-1.$$
And for any $x\in B_{\frac{1}{8}}$,
$$|(h-\tilde{p})(x)|\leq C|x|^3.
$$
Therefore for each $0<\lambda<\frac{1}{8}$, we have
\begin{eqnarray*}
\frac{1}{|B_{\lambda}|}\int_{\Omega_{\lambda}}|u-p|^{2}dx
&\leq&\frac{4}{|B_{\lambda}|}\int_{\Omega_{\lambda}}|u-h|^{2}dx
+\frac{4}{|B_{\lambda}|}\int_{\Omega_{\lambda}}|h-\tilde{p}|^{2}dx+\frac{4}{|B_{\lambda}|}\int_{\Omega_{\lambda}}|\frac {a}{2} x'^TSx'-\frac{a}{2}tr(S)x_{n}^{2}|^{2}dx\\
&\leq&\frac{4\varepsilon^2}{|B_{\lambda}|}+4C^{2}\lambda^{6}+4C^{2}\lambda^{4}\|S\|^{2}.
\end{eqnarray*}
Now for any $0<\alpha<1$, we take $\lambda$ small enough such that
$$4C^{2}\lambda^{6}\leq\frac{1}{3}\lambda^{2(2+\alpha)},
$$
and further we take $\varepsilon$ small sufficiently such that
$$\frac{4\varepsilon^{2}}{|B_{\lambda}|}\leq\frac{1}{3}\lambda^{2(2+\alpha)},
$$
and choose $\delta_{0}=\min\{\delta(\varepsilon),\frac{\sqrt{3}}{18C}\lambda^{\alpha}\}$ where $\delta(\varepsilon)$ is in Lemma $\ref{lm5}$, such that
$$9\lambda^{4}\|S\|^{2}\leq\frac{1}{3}\lambda^{2(2+\alpha)}.
$$
 Thus  Lemma follows.
\end{proof}

{\bf{Proof of Theorem $\ref{thm3}$:}}
Without loss of generality we assume that $f(0)=0$, $g(0)=0$, $Dg(0)=0$, and $D^2g(0)=0$. Otherwise, we may take $v(x)=u(x)-\frac {f(0)}{2n}|x|^2- g(0)-Dg(0)\cdot x-\frac 12 x^T D^2g(0)x $.  Let $\delta_0$ be as in Lemma $\ref{key3}$. We can assume the normalization conditions:$$ \displaystyle{\frac{1}{|B_1|}}\int_{\Omega_{1}}u^{2}dx\leq 1,~
 [f]^{2}_{\mathcal{L}^{2,\alpha}(0)}\leq\delta_{0}^{2}, ~  [g]_{C^{2,\alpha}(0)}\leq \frac{\delta_0}{2},~ K_\Omega\leq \frac{\delta_0}{C_0}, \|S\|\leq \frac{\delta_0}{C_0},$$
where $C_0$ is a constant to be determined later. Otherwise, we set
$$v(y)=\displaystyle\frac{u(x)}{\left(\displaystyle\frac{1}{|B_{1}|}\int_{\Omega_{1}}u^{2}dx
\right)^{\frac{1}{2}}+\displaystyle\frac{1}{\delta_{0}}\left([f]_{\mathcal{L}^{2,\alpha}(0)}+2[g]_{C^{2,\alpha}(0)}\right)}
\triangleq\displaystyle\frac{u(x)}{N},
$$  Where $y=\frac x{R}$ and $x\in B_{1}\cap\Omega$. By choosing $R$ small enough which depends only on $n$ and $K_\Omega$ and $\|S\| $, the normalization conditions hold. Without loss of generality, we assume that $R=1$.

Now, we prove the following claim inductively: there are harmonic polynomials
$$p_k(x)=a^{k}x_n+b_{in}^{k}x_{n}x_{i}-\frac 12 a^{k}x'^TSx'+\frac{1}{2}a^{k}tr(S)x_{n}^{2}$$
such that
\begin{equation}\label{ite}
\left(\frac{1}{|B_{\lambda^{k}}|}\int_{\Omega_{\lambda^{k}}}|u-p_k|^{2}dx\right)^{\frac{1}{2}}\leq \lambda^{k(2+\alpha)},\quad \forall k\geq0.
\end{equation}
and $a^0=0$, $b^0_{in}=0$, and for $k\geq 1$,
\begin{align}\label{co2}
    |a^k-a^{k-1}|&\leq C_0 \lambda^{(\alpha+1) (k-1)}, \\
    |b_{in}^k-b_{in}^{k-1}|&\leq C_0 \lambda^{\alpha{(k-1)}}.
    \end{align}
By the normalized assumption, we know that $k=0$ is just the assumption and the claim holds for $k=1$  since Lemma $\ref{key3}$.
Now let us assume  that the conclusion is true for $k$. Let
$$w(x)=\frac{(u-p_k)(\lambda^{k}x)}{\lambda^{k(2+\alpha)}},\quad \tilde{\Omega}=\{x:\lambda^{k}x\in \Omega\}.
$$
It is easy to check that $w$ is a weak solution to
\begin{eqnarray*}
\left\{
\begin{array}{rcll}
-\Delta w&=& \tilde{f}\qquad&\text{in}~~B_{1}\cap\tilde{\Omega},\\
\tilde{w}&=&\tilde{g}\qquad&\text{on}~~B_{1}\cap\partial\tilde{\Omega}, \\
\end{array}
\right.
\end{eqnarray*}
where $\tilde{f}(x)
=\displaystyle\frac{f(\lambda^{k}x)}{\lambda^{k\alpha}}$, $\tilde{g}(x)
=\displaystyle\frac{(g-p_k)(\lambda^{k}x)}{\lambda^{k(2+\alpha)}}$.
Thus by using the inductive assumption, we obtain that
$$\frac{1}{|B_{1}|}\displaystyle\int_{\tilde{\Omega}_{1}}w^{2}dx=
\frac{1}{|B_{1}|}\displaystyle\int_{\tilde{\Omega}_{1}}\frac{|u-p_k)(\lambda^{k}x)|^{2}}{\lambda^{2k(2+\alpha)}}dx
=\frac{\displaystyle\frac{1}{|B_{\lambda^{k}}|}\displaystyle\int_{\Omega_{\lambda^{k}}}|u-p_k|^{2}dx}{\lambda^{2k(2+\alpha)}}
\leq1,
$$
$$\frac 1{|B_1|}\int_{\tilde{\Omega}_{1}}\tilde{f}^{2}dx=\frac 1{|B_1|}
\int_{\Omega_{1}}\frac{|f(\lambda^{k}x)|^{2}}{\lambda^{2k\alpha}}dx=\frac 1{\lambda^{2k\alpha}|B_{\lambda^k}|}
\int_{\Omega_{\lambda^k}}|f(y)|^{2}dy
\leq [f]^{2}_{\mathcal{L}^{2,\alpha}(0)}
\leq\delta_{0}^{2},
$$
$$K_{\tilde{\Omega}}= \lambda^{k(1+\alpha)}K_{\Omega}\leq \frac{\delta_0}{C_0}.$$
Notice that by definition and the normalization conditions, we have for any $x\in \partial_w\Omega_1$
$$|x_n-\frac{1}{2}x'^TSx'|\leq [\partial\Omega]_{C^{2,\alpha}(0)}|x'|^{2+\alpha}\leq \frac {\delta_0}{C_0}|x'|^{2+\alpha},$$
and
$$|x_n|\leq \frac{C\delta_0}{C_0}|x'|^{1+\alpha}.$$
Now for $x\in \partial \tilde{\Omega}$
From ($\ref{co2}$), we  have   that $|a^{k}|+|b^{k}|\leq C$ for some $C>0$. Then we yield that
\begin{eqnarray*}
\|\tilde{g}\|_{L^{\infty}(\partial\tilde\Omega\cap B_1)}&=&\|\frac{(g-p_k)(x)}{\lambda^{k(2+\alpha)}}\|_{L^{\infty}(\partial\Omega\cap B_{\lambda^{k}})}\\
&\leq& \|\frac{g(x)}{\lambda^{k(2+\alpha)}}\|_{L^{\infty}(\partial\Omega\cap B_{\lambda^{k}})}+\|\frac{a^{k}(x_n-\frac{1}{2}x'^TSx')+b^{k}_{in}x_{n}x_{j}+\frac{1}{2}a^{k}tr(S)x_{n}^{2}}{\lambda^{k(2+\alpha)}}\|_{L^{\infty}(\partial\Omega\cap B_{\lambda^{k}})}\\
&\leq& \frac 12 [g]_{C^{2,\alpha}(0)}+\frac{C\delta_0}{C_0}\\
&\leq & \frac 12 \delta_{0}+\frac 12 \delta_{0}=\delta_0,
\end{eqnarray*}
if we choose $C_0$ such that $\frac{C}{C_0}\leq \frac 12$. Now we can apply Lemma $\ref{key3}$ for $w$ to obtain that
$$\left(\frac{1}{|B_{\lambda}|}\int_{\tilde{\Omega}_{\lambda}}|w-\bar{p}|^{2}dx\right)^{\frac{1}{2}}\leq\lambda^{2+\alpha}
$$ where $\bar{p}(x)=ax_{n}-\frac {a}{2} x'^TSx'+b_{in}x_{n}x_{i}+\frac{a}{2}tr(S)x_{n}^{2} $ with
$|a|+ |b_{in}|\leq C$ for $x\in \tilde{\Omega}_\lambda$.  We scale back to get
$$\left(\frac{1}{|B_{\lambda^{k+1}}|}\int_{{\Omega}_{\lambda^{k+1}}}|u(x)-p_{k}(x)-\lambda^{(2+\alpha)k}\bar{p}(\frac{x}{\lambda^k})|^{2}dx\right)^{\frac{1}{2}}
\leq\lambda^{(k+1)(2+\alpha)} .
$$
Thus we prove the $(k+1)$-th step by choosing $p_{k+1}(x)=p_{k}(x)+\lambda^{(2+\alpha)k}\bar{p}(\frac{x}{\lambda^k})$. It is clear that
$$|a^{k+1}-a^k|\leq \lambda^{(\alpha+1) k}|a|\leq C\lambda^{(\alpha+1) k},$$
and
 $$|b_{in}^{k+1}-b_{in}^{k}|\leq C \lambda^{\alpha k}.$$
It is also clear that ${a^k}$ and $b^k_{in}$ is Cauchy sequence and converge to $a_\infty$ and $b_{\infty,in}$. The limiting polynomial is $$ p_\infty(x)=a^\infty x_n+b^{\infty}_ {in}x_{n}x_{i}-\frac 12 a^{\infty}x'^TSx' +\frac{1}{2}a^{\infty}tr(S)x_{n}^{2} $$ satisfies
$$|p_k(x)-p_\infty(x)|\leq C\lambda^{k(2+\alpha)},~~~\text { for any } |x|\leq \lambda^k. $$
Hence,  for any $r\leq 1$, there is a k with $\lambda^{k+1}\leq r\leq \lambda^{k}$, and that
\begin{eqnarray*}
\frac{1}{|B_{r}|}\int_{{\Omega}_{r}}|u(x)-p_\infty|^{2}dx&\leq&\frac{1}{\lambda^{n}|B_{\lambda^{k}}|}\int_{{\Omega}_{\lambda^{k}}}|u(x)-p_{k}(x)|^{2}dx+\frac{1}{\lambda^{n}|B_{\lambda^{k}}|}\int_{{\Omega}_{\lambda^{k}}}|p_\infty(x)-p_{k}(x)|^{2}dx\\
&\leq& \frac{\lambda^{2k(2+\alpha)}}{\lambda^{n}}+\frac{C\lambda^{2k(2+\alpha)}}{\lambda^{n}}\\
&\leq& Cr^{2(2+\alpha)}
.
\end{eqnarray*}
This implies that $u$ is $C^{2,\alpha}$ at 0 in the $L^{2}$ sense, and
$$
[u]_{\mathcal{L}^{2,2+\alpha}(0)}\leq C\left(\displaystyle\frac{1}{|B_{1}|}\int_{\Omega_{1}}u^{2}dx
\right)^{\frac{1}{2}}+C\left(||f||_{\mathcal{L}^{2,\alpha}(0)}+||g||_{C^{2,\alpha}(0)}\right).
$$
Thus we complete the proof of Theorem \ref{thm3}.

\section*{Data availability }	Data sharing not applicable to this article as no datasets were generated or analyzed during
 the current study.

\section*{Declarations}  No potential conflict of interest was reported by the authors.

\end{document}